\documentclass[12pt]{amsart}
\usepackage{amsmath}
\usepackage[english]{babel}
\usepackage[colorlinks,
linkcolor=blue,
urlcolor=blue,
citecolor=red]{hyperref}
\usepackage[usenames,dvipsnames]{xcolor}
\usepackage{microtype}

\newcommand{\R}{{\mathbf R}}

\newcommand{\ep}{{\epsilon}}
\newcommand{\spt}{\mathop{\rm spt}}

\newcommand{\p}{{\partial}}

\newcommand{\Hess}{\mathop{\rm Hess}}

\newcommand{\LL}{\mathcal{L}}

\newcommand{\vol}{{\rm vol}}
\newcommand{\vg}{\vol_g}

\newcommand{\pt}{{\p}}
\newcommand{\Ric}{{R}}
\renewcommand{\d}{{d}}

\newtheorem{theorem}{Theorem}

\newtheorem{corollary}[theorem]{Corollary}

\newtheorem{lemma}[theorem]{Lemma}

\newtheorem{proposition}[theorem]{Proposition}
\newtheorem{remark}[theorem]{Remark}

\newdimen\dummy
\dummy=\oddsidemargin
\usepackage{todonotes}

\begin{document}

\title[An elliptic proof of the Lorentzian splitting theorems]
{An elliptic proof of the splitting theorems from Lorentzian geometry}
\author[Braun]{Mathias Braun}
\address{Institute of Mathematics, EPFL, 1015 Lausanne, Switzerland\newline \tt mathias.braun@epfl.ch}
\author[Gigli]{Nicola Gigli}\address{SISSA, Trieste, Italy\newline \tt ngigli@sissa.it}
\author[McCann]{Robert J. McCann}
\address{Departments of Mathematics and Economics, University of Toronto, Toronto, Ontario, Canada\newline \tt mccann@math.toronto.edu}
\author[Ohanyan]{Argam Ohanyan}
\address{Department of Mathematics, University of Toronto, Toronto, Ontario, Canada\newline \tt argam.ohanyan@utoronto.ca}
\author[S\"amann]{Clemens S\"amann}
\address{{Faculty of Mathematics, University of Vienna, Vienna, Austria\newline \tt clemens.saemann@univie.ac.at}}

\thanks{2020 MSC Classification Primary: 83C75 Secondary: 35J92 35Q75 (49Q22 51K10) 53C21 53C50 58J05.} 


\thanks{\copyright \today\ by the authors. Under revision for \em Comm. Amer. Math. Soc.}
\begin{abstract}
We initiate the development of a theory of the negative homogeneity $p$-d'Alembert operator in the smooth setting. We relate it to a Bochner-Ohta identity of homogeneity of $2p-2<0$.
We identify conditions under which the unexpected ellipticity of this operator turns out to be uniform.
We exploit 
this to give
a new proof of the Eschenburg, Galloway and Newman splitting theorems from Lorentzian geometry, which allows us bring them into a framework closer to the Cheeger-Gromoll splitting theorem from Riemannian geometry.  
\end{abstract}

\maketitle

\tableofcontents

\section{Introduction}

\subsection{History and results}
Splitting theorems play a vital role in both Riemannian \cite{CheegerEbin98} and Lorentzian~\cite{BeemEhrlichEasley96} geometry.  Under the strong energy condition from general relativity,
they confirm the intuition expressed, e.g.\ by Geroch~\cite{Geroch70}, following the discovery of the singularity theorems \cite{Penrose65a,Hawking66a,HawkingPenrose70} by Penrose and Hawking,
that a geodesically complete spacetime ought to be exceptional: if even one of its complete geodesics is timelike and maximizing, then the space is a stationary, static, geometric product. 
This is made precise by the following theorem, 
in which spacetime refers to a connected smooth $n$-dimensional manifold $M$ carrying a smooth Lorentzian metric tensor $g_{ij}$ 
of signature $(+,-,\ldots,-)$
and a continuous timelike vector field which distinguishes future from past. 
Our methods are adapted to obtain Lorentzian splitting theorems for low regularity metrics $g_{ij} \in C^{1}_{loc}(M)$ with the option of Bakry-\'Emery style weights \cite{Case10} in a companion work \cite{BGMOS25+}.
The present manuscript focuses on the classical smooth case to exploit existing theory \cite{GallowayHorta96,McCann20}, and to prevent technicalities from obscuring the power and simplicity of the new ideas.

\begin{theorem}[Lorentzian splitting theorem]
\label{T:Lorentzian splitting}
If a spacetime $(M,g)$ satisfies the strong energy condition 
\begin{equation}\label{strong energy}
R(v,v) \ge 0\ \mbox{\rm for all timelike vectors}\ v,
\end{equation} 
contains a proper-time maximizing, isometrically embedded copy $\gamma$ of the Euclidean line $\R$,
and is either   (a) globally hyperbolic or (b) timelike geodesically complete,
then $(M,g)$ is isometric to $(\R \times S, dt^2 - h)$ where $(S,h)$ is a complete Riemannian manifold with nonnegative Ricci curvature.
\end{theorem}

For smooth metric tensors $g$, the theorem was conjectured under {timelike geodesic completeness} 
(b) by Yau \cite{Yau82} and proved by Newman \cite{Newman90},
after Galloway had established the conclusion for spacetimes with compact Cauchy surfaces \cite{Galloway84},
Beem, Ehrlich, Markvorsen and 
Galloway~\cite{BeemEhrlichMarkvorsenGalloway84}
under (a) plus the stronger hypothesis of a suitable sign on all timelike {\em sectional} curvatures,  Eschenburg \cite{Eschenburg88} under (a) plus (b),  and 
Galloway~\cite{Galloway89} under (a) using a result of Bartnik~\cite{Bartnik88b}.  
Under a {\em nonsmooth} sectional curvature hypothesis relaxing \cite{BeemEhrlichMarkvorsenGalloway84}, the result has been extended to the Lorentzian length space \cite{KunzingerSaemann18} setting
by Beran, Ohanyan, Rott and Solis~\cite{BeranOhanyanRottSolis23}. Given recent interest in establishing Hawking and/or Penrose type singularity theorems
in very low regularity settings \cite{KunzingerSteinbauerStojkovicVickers2015,KunzingerSteinbauerVickers15,GrafGrantKunzingerSteinbauer18,Graf20,KunzingerOhanyanSchinnerlSteinbauer2022,AlexanderGrafKunzingerSaemann19+,CavallettiMondino20+},
 the generalization of Theorem~\ref{T:Lorentzian splitting} to such a setting remains a tantalizing challenge 
for future research; cf.~\cite{BeranOctet24+} and analogous developments in positive signature \cite{Gigli26,Gigli14,Gigli15}.

Existing proofs are rather complicated relative to that of the analogous 
 splitting theorem from Riemannian geometry by
Cheeger and Gromoll~\cite{CheegerGromoll71} or its simplification by Eschenburg and Heintze \cite{EschenburgHeintze84}.
The reason for this is that the Lorentzian Laplacian
(also known as the {\em d'Alembertian} or wave operator), although linear, is hyperbolic rather than elliptic.  The purpose of the present article is to simplify the proof of the Lorentzian splitting theorems, by replacing the Laplacian with the nonlinear $p$-d'Alembert operator $-\square_p u := \nabla \cdot (|du|^{p-2} du) = \delta E/\delta u$ given by the variational
derivative of the functional $E(u) := \int_M H(du) \, d\vg$ with
\begin{equation}\label{Hamiltonian}
H(v):=
\begin{cases}
-\frac1p |v|^p := -\frac1p (g^{ij}v_iv_j)^{p/2} & v\ \mbox{is future-directed},\ p \ne 0, \\
-\log |v| & v\ \mbox{is future-directed},\ p=0, \\
+\infty & \mbox{else.}
\end{cases}
\end{equation}
In the range $p<1$ this operator
turns out to be (non-uniformly) elliptic~\cite{BeranOctet24+} as a consequence of convexity of the Hamiltonian integrand $H(v)$ introduced for $p\ne 0$  in \cite{McCann20,Min:19}. 
Thus we sacrifice linearity to gain ellipticity. 
Our signs have been chosen to make the linearization of 
$\square_p$ around any suitable $u$ with $E(u)<\infty$ a negative semidefinite operator, in analogy with 
the analyst's rather than the (Riemannian) geometer's Laplacian.
Combining this insight with optimal transport ideas,
elliptic techniques \cite{GilbargTrudinger98}, 
a Bochner-Ohta identity of homogeneity $2p-2<0$, 
cf.\ \cite{Ohta14h,MondinoSuhr23}, and the simplification of Eschenburg's and Newman's arguments by Galloway and Horta \cite{GallowayHorta96},  we are able to arrive at a 
{conceptually new}
proof of the splitting theorem.
We give a brief outline of our strategy below.
Curiously, this new approach to the smooth theorem 
relies on a d'Alembert comparison estimate for $p<1$ 
first established in a much less smooth setting with the octet~\cite{BeranOctet24+}.

\subsection{Overview of strategy}

The geometry of a spacetime is well-known to be encoded in a time-separation (or Lorentzian distance) function $\ell:M^2 \longrightarrow \{-\infty\} \cup [0,\infty]$ defined by
a Lagrangian action principle
\begin{align}
\label{global time separation}
\ell(x,y) &:= \sup_{\sigma(a) = x, \sigma(b)=y} \LL(\sigma),
\\ \LL(\sigma) &:= \int_a^b g(\sigma'(r),\sigma'(r))^{1/2} dr,
\label{action}
\end{align}
where the supremum is over future-directed Lipschitz curves $\sigma$ connecting $x,y \in M$.
It satisfies,
for all $x,y,z \in M$, the reverse triangle inequality
\begin{equation}\label{rti}
\ell(x,z) \ge \ell(x,y) + \ell(y,z),
\end{equation}
{(with the convention $\infty-\infty=-\infty$ under hypothesis (b)).}
(Under hypothesis (a), it also satisfies the reflexivity property $\ell(x,x)=0$ and the antisymmetry condition
$$
\min\{\ell(x,y),\ell(y,x)\} =-\infty\ \mbox{unless}\ x=y.)
$$
The chronological and causal relations are then defined by $I^+= \{ \ell >0\} \subset M^2$ and $J^+=\{ \ell \ge 0\} \subset M^2$, 
and we also write any of the equivalent conditions (i) $x \ll y$ or (ii) $y \in I^+(x)$ or (iii) $x \in I^-(y)$ if and only if (iv) $(x,y) \in I^+$;  similarly, we write (i') $x \le y$ or (ii') $y \in J^+(x)$ or (iii') $x \in J^-(y)$ if and only if (iv') $(x,y) \in J^+$. Also $I^\pm(X) := \cup_{x \in X} I^\pm(x)$ and $I(X,Y) :=I^+(X)\cap I^-(Y)$ for $X,Y \subset M$ and analogously for $J^\pm$. Note that, on any spacetime, the reverse triangle inequality implies the following push-up property for the timelike relation: whenever $x \leq y \ll z$ or $x \ll y \leq z$, then already $x \ll z$, for all $x,y,z \in M$.

Given $[a,b] \subset [-\infty,\infty]$,
a {\em ray} will refer to an inextendible, 
maximizing, causal geodesic $\gamma:[a,b)\longrightarrow M$.
Here {\em inextendible} means $\lim_{s\to b} \gamma(s)$ does not converge to any point in $M$,
and {\em maximizing} means the Lorentzian length functional 
\eqref{action} satisfies
$\LL(\gamma|_{[a,s]})=\ell(\gamma(a),\gamma(s))$
for each $s\in [a,b)$.
Being a causal geodesic,  $\gamma$ is affinely parameterized by default and either future- or past-directed; we say $\gamma$ is ({\em future} or {\em past}) {\em complete} precisely when $b=\infty$ in this parameterization.
A {\em line} will refer to a doubly inextendible, 
maximizing, causal geodesic $\gamma:(a,b) \longrightarrow M$,
where now maximizing means $\LL(\gamma|_{[r,s]})=\ell(\gamma(r),\gamma(s))$
for each $[r,s]\subset (a,b)$.
When future-directed, the (affinely parameterized) line is said to be {\em future-complete} if $b=\infty$, {\em past-complete} if $a=-\infty$, and {\em complete} if both hold. For causal lines and rays, inextendibility follows from completeness assuming (a) or (b).  For timelike rays 
{under timelike geodesically completeness} (b) only,  the converse is true.
Thus one of the technical challenges surmounted by Galloway working under (a) {global hyperbolicity} without (b) was the fact that rays and lines need not be complete, because the exponential map may not be globally defined \cite{Galloway89}.  Working under (b) without (a), Newman instead had to contend with $(x,y)$ for which the supremum \eqref{global time separation} defining 
$\ell(x,y) \ge 0$ need neither be attained nor continuous nor finite \cite{Newman90}.
Henceforth, when we speak of lines or rays, we will mean future-directed, timelike, and proper-time parameterized, unless explicitly stated otherwise.
With this convention,  any (line or) ray defined on a (doubly) unbounded domain may be inferred to be complete 
unless otherwise noted.
We write $I^\pm(\gamma):= I^\pm(\gamma((a,b))$ and $I(\gamma):= I^+(\gamma)\cap I^-(\gamma)$ for a line, and similarly for a ray.

As in more traditional approaches to the splitting theorems (except \cite{Galloway84} and \cite{Bartnik88}), the central objects of interest are the forward and backward Busemann functions $\pm b^\pm$ analogous to those introduced in \cite{Busemann32} and rediscovered by Cheeger and Gromoll \cite{CheegerGromoll71}.  The Busemann function 
$b^+:M \longrightarrow [-\infty,\infty]$ associated to a 
complete future-directed ray $\gamma$ is defined as the limit $b^+(x)=\lim_{r\to \infty} b^+_r(x)$,  where
\begin{equation}\label{b^+_r}
b^+_r(x) = \ell(\gamma(0),\gamma(r))-\ell(x,\gamma(r)); 
\end{equation}
the negated Busemann function of a complete past-directed ray is defined by $b^-(x) = \lim_{r\to-\infty}b^-_r(x)$, where
\begin{equation}\label{b^-_r}
b^-_r(x) = \ell(\gamma(r),x)-\ell(\gamma(r), \gamma(0)).
\end{equation}
Thus a complete future-directed line $\gamma$ has two Busemann functions $\pm b^\pm$ associated with it.  From the triangle inequality \eqref{rti} and proper-time parameterization of $\gamma$ it easily follows and is well-known that the limits above converge monotonically and $r>0$ implies
\begin{equation}\label{ordering}
b^+_r \ge b^+ \ge b^- \ge b^-_{-r},
\end{equation}
with all four functions coinciding on the intersection of the line $\gamma$ with the open diamond $I(\gamma(-r),\gamma(r))$.
Note that $db^\pm$ does not depend on which point on the line $\gamma$ is chosen as the base $\gamma(0)$,
because the limiting functions $b^+$ differ only by additive constants.

In the linear case $p=2$, where $b^+_r$ is smooth the strong energy condition \eqref{strong energy}
is well-known to yield the following bound \cite[\S 5]{Eschenburg88} on the d'Alembertian

\begin{equation}\label{p-comparison}
\square b^+_r = \square_p b^+_r \le \frac{n-1}{\ell(\cdot,\gamma(r))} \quad {\rm on}\ I^-(\gamma(r));
\end{equation}
the same inequality extends across the timelike cutlocus when derivatives are interpreted in a suitable weak sense \cite{BeranOctet24+}, 
as in the Riemannian case~\cite{Calabi58v}.
Since $|db^+_r|=1$,  it is perhaps not surprising that we are able to show the same inequality holds weakly for the $p$-d'Alembertian if $p<1$;  in the current smooth setting, we give a new proof of this fact that is logically independent of our original argument with the 
octet~\cite{BeranOctet24+}.
With our sign conventions, the Busemann functions $b^+\ge b^-$ associated to a line are therefore $p$-super- and $p$-subharmonic:
$\square_p b^+ \le 0 \le \square_p b^-$. 
For $p<1$ however, (non-uniform) ellipticity of $-\square_p$ and the strong maximum principle suggest the touching super- and subsolutions
must coincide,  so that $b^+=b^-$ is $p$-harmonic. 
In that case the negatively homogeneous Bochner-Ohta formula established below (which 
{is a variant of formulas appearing in \cite{Ohta14h, MondinoSuhr23}})
implies $b^+$ is smooth and its Hessian vanishes identically; by contrast, the more familiar linear/quadratic Bochner-Weitzenb\"ock identity is unhelpful since in Lorentzian signature
it controls only a signed difference of Hessian terms (analogous to the failure of the usual d'Alembertian to be elliptic).
At this point we are essentially done:  the level sets $S_r: = \{x \in M \mid b^+(x)=r\}$ of $b^+$ are spacelike surfaces whose normal vector $N=d b^+$ is locally parallel.  From here it is easy to deduce the second fundamental form of $S_r$ vanishes and that $N$ is a Killing vector field generating a local isometry between $S_r$ and $S_0$ for each $r \in \R$.  An argument is still needed to globalize this result, achieved through a simplification of \cite{Eschenburg88,Galloway89,Newman90}.

\subsection{Technical challenges} Unfortunately, the strong maximum principle is delicate to establish for non-uniformly elliptic nonlinear equations; see e.g.~the more standard case $p>1$ treated in 
\cite{GoffiPediconi21,PucciSerrin04} and their 
references for the $p$-Laplacian in Riemannian signature; 
examples of Krol and Lewis discussed in \cite{Lewis80} show the $C^{1,\alpha}$ regularity established, e.g.\ in \cite{diBenedetto83,Lewis83}, is best possible when $p>2$.  Our equation with $p<1$ in Lorentzian signature is worse.
We must first establish that the equation becomes uniformly elliptic when linearized around the Busemann functions in question. To do this requires more regularity than Busemann functions generally possess.
(As observed  in~\cite{GallowayHorta96}, de-Sitter space is both (a) globally hyperbolic and (b) (timelike) geodesically complete, but its Busemann functions exhibit discontinuities and are not globally real-valued.) However, under the hypotheses of Theorem~\ref{T:Lorentzian splitting},  Galloway and Horta gave a simple demonstration of Eschenburg's observation that the Busemann functions are Lipschitz continuous in a neighbourhood of the line.  By combining their approach with Proposition~3.4 of McCann~\cite{McCann20}, we are able to improve this conclusion by showing the limiting Busemann function $b^+$ inherits semiconcavity from the approximate Busemann functions $b^+_r$.   
In this neighbourhood then, we can pass to the limit of the (weakly reformulated) 
nonlinear comparison inequality \eqref{p-comparison} and show the linearization of the resulting operator around the Busemann function becomes uniformly elliptic. Unlike the Riemannian case $p=2$, the weak form of \eqref{p-comparison} involves only one integration by parts so to obtain the limit $r \to \infty$ requires a.e.\ convergence not only of $b_r^+$ to $b^+$ but also of $db_r^+$ to $db^+$. 
This is enough regularity to obtain $b^+=b^-$ from the strong maximum principle; moreover, it is enough regularity to apply our Bochner-Ohta identity and
 conclude that the geometry splits locally and smoothly on the neighbourhood in question. {In comparison, we mention a maximum principle by Andersson-Galloway-Howard \cite{AnderssonGallowayHoward98}. Their idea (generalizing the same line of thought from the original proofs of the splitting theorem) is that after restricting the usual d'Alembertian to a foliation of certain spacelike hypersurfaces, it becomes the respective intrinsic Riemannian Laplacian, for which elliptic maximum principle techniques apply. Our approach is global in nature and does not require this restriction process. This comes with several advantages. For instance, t}{o globalize this splitting theorem we are able to simplify the strategies of \cite{Eschenburg88,GallowayHorta96} by using the knowledge that the Hessian of $b^+=b^-$ vanishes on a neighbourhood of the line,  and \emph{not just on a spacelike hypersurface}.
 We still need to propagate this information outside the neighbourhood in question,  first along nearby asymptotes to the line, and then throughout $M$ by connectedness. 
 }

\subsection{Outlook}

In this work we focus on developing elliptic $p$-d'Alembert methods and using them to significantly simplify the proofs of the classical Lorentzian splitting theorems \cite{Eschenburg88, Galloway89, Newman90, GallowayHorta96}. In the more complicated setting of our companion work, we exploit the same techniques to obtain new generalizations both to weighted spacetimes and to Lorentzian metric tensors whose regularity is merely $C^1_{loc}$ \cite{BGMOS25+}; see~\cite{KunzingerOhanyanVardabasso25} for an analogous splitting statement in Riemannian signature. The rich literature on spacetime geometry with non-smooth Lorentzian metrics \cite{KunzingerSteinbauerStojkovicVickers2015, KunzingerSteinbauerVickers15, GrafGrantKunzingerSteinbauer18, Graf20, KunzingerOhanyanSchinnerlSteinbauer2022} highlights the relevance of such results, as non-smooth metrics arise naturally in general relativity, e.g.\ as solutions to the Einstein equations.\ Further refinements of the Lorentzian splitting theorems may require a more sophisticated analysis of the $p$-d'Alembert operator, which
we defer to future work.

\section{Equi-semiconcavity of the Busemann limits}

A function on a Lorentzian manifold will be called {\em linear} if its covariant Hessian vanishes. Our goal is to prove linearity of the Busemann function $b^+$.  
In contrast to the Riemannian setting,  where the set of linear functions of slope 1 which vanish at a given point is compact,
this is not the case in Lorentzian geometry,  as the limits $\theta \to \pm \infty$ of the following elementary family of functions on the Minkowski $(t,x)$ plane show: 
\begin{equation}\label{divergent linear family}
\{t \cosh \theta + x \sinh \theta\}_{\theta \in \R}.
\end{equation}
This divergence is due to noncompactness of the pseudosphere,
and reflects the non-uniform ellipticity of the $p$-d'Alembert operator for $p<1$. A similar mechanism is responsible
for the poor regularity of Lorentzian Busemann functions more generally.  However, in a neighbourhood of the line defining the Busemann function,  this non-uniformity of ellipticity will be ruled out using a series of simple but subtle observations by Eschenburg \cite{Eschenburg88} and Galloway \cite{Galloway89} 
{under globally hyperbolicity} (a) and by Galloway and Horta \cite{GallowayHorta96} following Newman \cite{Newman90} 
{under timelike geodesically completeness} (b),
summarized in the following theorem.
Here Lipschitz refers to any auxiliary Riemannian metric $\tilde g$ on $M$.  Such an auxiliary metric plays a useful role throughout,
{and can be taken to be complete on $M$ \cite{NomizuOzeki61,Geroch68}.}

\begin{theorem}[Busemann functions are Lipschitz near the line]\label{T:Lipschitz}
In an (a) globally hyperbolic or (b) timelike geodesically spacetime $M$, let $\gamma: 
(-\infty,\infty) \longrightarrow M$ be a complete future-directed timelike line.
Then (i) the Busemann functions $b^+$ and $-b^-$ associated with $\gamma$ are locally Lipschitz continuous in a neighbourhood $U$ of $\gamma$. 
Given $X \subset U$ compact,  there exists $R,
\tilde C >0$ 
such that  such that (ii) a maximizing timelike proper-time parameterized geodesic segment connects $x$ to $\gamma(r)$ for each pair
$x\in X$ and $|r| \ge R$. Moreover, (iii) for $x \in X$ any such segment $\sigma$ satisfies
$\dot \sigma(0) \in K_{\tilde C}(x) := \{v \in T_x M \mid |v|_g =1, |v|_{\tilde g} \le \tilde C\}$.
\end{theorem}

\begin{proof}
Under {timelike geodesic completeness} (b),  Theorem 3.7 and Corollary 5.2 of \cite{GallowayHorta96} together yield the statement for a neighbourhood of $\{\gamma(r)\}_{r>0}$, assuming only that $\gamma$ is a ray.
However, in the case of a line the Busemann function $b^+$ is independent of which point on $\gamma$ is selected to be $\gamma(0)$, apart from an additive constant.
In case (b) this establishes (i); while (ii)--(iii) 
follow from Lemma 3.3--3.4 (ibid), completeness of $\gamma$, 
and compactness of $X$, using a covering argument. 
As Galloway and Horta 
note, (iii) and (i) are even easier to prove in {the globally hyperbolic} case (a),
as was first done by Eschenburg \cite[Lemmas 3.2--3.3]{Eschenburg88}
and exploited in \cite{Galloway89}; (ii) holds as soon as
$X \subset I^-(\gamma(R)) \cap I^+(\gamma(-R))$ in case (a),
and the existence of such $R$ follows from 
compactness of $X \subset U$ by a covering argument 
since finiteness of $b^\pm$ implies $U\subset I(\gamma)$.
\end{proof}

\begin{remark}[Conic intersections]\label{R:conic intersections}
Near any point in this neighbourhood,  the Lipschitz bound combines with the monotonicity (cf.~\eqref{1-steep}) of $b^+$ to 
prevent $db^+$ from escaping to infinity asymptotically to the light cone:   it must lie in the intersection of 
the solid future hyperboloid $g(v,v) \ge 1$ with the ellipsoid $\tilde g(v,v) \le \tilde C$.
Moreover, in Corollary \ref{C:semiconcave} below, we will show equality holds in $|db^+| \ge 1$, 
 hence $|db^+|=1=|db^+_r|$.
 \end{remark}

Heuristically at least, re-expressing the $p$-d'Alembertian in nondivergence form
\begin{align}
\begin{split}
\frac{\square_p u}{|du|^{p-2}}
&= -\frac1{|du|^{p-2}}\nabla \cdot (|du|^{p-2} du) 
\\&= (2-p)(\Hess u)(\frac{du}{|du|},\frac{du}{|du|}) +\square_2 u
\\&= \left[(2-p) \frac{\nabla^i u \nabla^j u}{|\nabla u|^2} - g^{ij} \right]\nabla_i\nabla_j u 
\label{nondivergence}
\end{split}
\end{align}
where $\nabla$ denotes the Levi-Civita connection of $g$, demonstrates this ellipticity.  Recalling the Lorentzian signature $(+, -, \cdots, -)$, by choosing suitable coordinates
we will eventually show that the term in square brackets 
--- presently expressed using abstract index notation \cite{Wald84} ---
becomes positive-definite for $p<1$ as soon as $du$ is timelike.  Indeed,  one could regard the term in square brackets as a (nonsmooth) Riemannian metric $h^+$ 
induced by $u=b^+$. 
 This suggests the associated metric tensor $h_{ij}^+ \in L^\infty(U)$ will become uniformly positive-definite in suitable coordinates --- 
 corresponding to uniform ellipticity of the nondivergence form 
 operator.   Notice however, since the coefficients are merely bounded and measurable,  the passage between divergence and nondivergence form is not automatic.

{We will frequently make use of \cite[Lemma 2.5]{GallowayHorta96}, which we shall prove now for non-smooth metrics (in anticipation of future applications). To this end, note that a (future) $S$-\emph{ray}, for some subset $S \subset M$, is a future inextendible causal geodesic $\gamma:[0,a) \to M$ maximizing the time separation to $S$, i.e.\ $\mathcal{L}(\gamma|_{[0,t]}) = \sup_{x \in S} \ell(x,\gamma(t))$, for every $t \in [0,a)$. Here, $\mathcal{L}$ denotes the Lorentzian arclength functional.
\begin{lemma}[Upper supports to approximate Busemann functions]
\label{L:upper support}
If $\gamma:[0,\infty)\longrightarrow M$ is a future complete $S$-ray in a spacetime $(M,g)$ with $g \in C^{0,1}_{\mathrm{loc}}$ and 
$\sigma_r:[0,s_r]\longrightarrow M$ is timelike and maximizing
between $\sigma_r(0)\in I^+(S)\cap I^-(\gamma(r))$ and $\sigma_r(s_r)=\gamma(r)$ for some $r>0$, then for each $s \in (0,s_r]$
\begin{equation}\label{u_r}
u_r(x)
:= b_r^+(\sigma_r(0)) + \ell(\sigma_r(0),\sigma_r(s)) - \ell(x,\sigma_r(s))
\end{equation}
satisfies $u_r(x) \ge b_r^+(x)$ for all $x \in M$,
where $b_r^+$ is from \eqref{b^+_r}. Equality holds at $x=\sigma_r(a)$ for each $a\in[0,s]$.
Moreover, both $u_r$ and $b^+_r$ are real-valued on 
the neighbourhood $I^+(S) \cap I^-(\sigma_r(s))$ of $\sigma_r(0)$.
\end{lemma}

\begin{proof}
The definition \eqref{b^+_r} of $b^+_r$, maximality of 
$\sigma_r:[0,s_r] \longrightarrow M$ between $\sigma_r(0)$ and $\sigma_r(s_r)=\gamma(r)$,
and $s \in (0,s_r]$ yield
\begin{align}
u_r(x) &= \ell(\gamma(0),\sigma_r(s_r)) - \ell(\sigma_r(0),\sigma_r(s_r)) +\ell(\sigma_r(0),\sigma_r(s))- \ell(x,\sigma_r(s))
\nonumber \\ &= \ell(\gamma(0),\sigma_r(s_r)) - \ell(\sigma_r(s),\sigma_r(s_r)) - \ell(x,\sigma_r(s))
\nonumber \\ & \ge \ell(\gamma(0),\sigma_r(s_r))- \ell(x,\sigma_r(s_r)) 
\nonumber \\ & =b^+_r(x)
\nonumber 
\end{align}
as desired,  where the reverse triangle inequality \eqref{rti} has been used. 

For $a \in [0,s]$, equality holds at $x=\sigma_r(a)$ 
since maximality of $\sigma_r$ ensures the triangle inequality is saturated.  The fact that $\gamma$
maximizes time from $S$ to $\gamma(r)$ implies
$0<\ell(x,\sigma_r(s_r)) \le r$ for each $x \in I^+(S)\cap I^-(\sigma_r(s_r))$. The reverse triangle inequality shows $\ell(x,\sigma_r(s))$, $u_r(x)$ and 
$b^+_r(x)$ are all finite
on the smaller neighbourhood $I^+(S)\cap I^-(\sigma_r(s))$ of $\sigma_r(0)$.
\end{proof}
}

We assume $g_{ij} \in C^\infty$ hereafter.
Our first task will be to improve the regularity of the Busemann function on the neighbourhood mentioned above by showing it is semiconcave.
Fixing a smooth Riemannian metric $\tilde g$ on $M$, recall a function $u:M \longrightarrow \R$ is said to be {\em semiconcave} on $U\subset M$, with {\em semiconcavity constant} $C<\infty$ if 
$$
\limsup_{w \to 0} \frac{u(\exp^{\tilde g}_x w) + u(\exp^{\tilde g}_x -w) - 2u(x)}{\tilde g(w,w)} \le C
$$
for each $x \in U$. It is said to be locally semiconcave on $U$ if it is semiconcave on each compact subset of $U$.  Here $\exp^{\tilde g}$ denotes the Riemannian exponential map.   Although $C$ depends on $\tilde g$,  the property of being locally semiconcave does not.
For each $y \in M$,  the function $v(x) = -\ell(x,y)$ was shown to be semiconcave near each $x \in I^-(y)$ in Proposition 3.4 of \cite{McCann20}, with a semiconcavity constant $C(x,y)$ depending continuously on $x$ and $y$.   
The following lemma shows this semiconcavity is inherited by the Busemann function $b^+$.
A family of functions such as $\{b^+_r\}_{r \ge R}$ is said to be {\em equi-semiconcave} on $U \subset M$ if they share the same semiconcavity constant on each compact subset $X$ of $U$.

\begin{proposition}[Equi-semiconcavity of Busemann limits near the line]\label{P:semiconcave}
Let $b^+=\lim_{r \to \infty} b^+_r$ denote the Busemann function associated by \eqref{b^+_r} to a complete future-directed timelike line  $\gamma:(-\infty,\infty)\longrightarrow M$.
Then $b^+$ is locally semiconcave on the neighbourhood $U$ of $\gamma$ provided by Theorem \ref{T:Lipschitz}(a) or (b). 
On each compact set $X \subset U$, 
one can find $R=R(M,g, \tilde g,X,\gamma,\gamma(0))$ and a single semiconcavity constant $C=C(M,g,\tilde g,X,\gamma(\R))$ which works throughout $X$ for all $b^+_r$ with $r \ge R$. 
\end{proposition}

\begin{proof}
For $R$ large enough (depending on $\gamma(0)$) we shall show
 $(b^+_r)_{r \ge R}$ has a semiconcavity constant $C$ of which depends on $\gamma$ but not on $\gamma(0)$.
Fix a compact set $X \subset U$. Theorem \ref{T:Lipschitz}(ii) provides $R$ and $\tilde C$ such that for each $x \in X$ and $z =\gamma(r)$ with $r \ge R$,
there is a 
 (proper-time parameterized) maximizing geodesic segment $\sigma_r$ joining $\sigma_r(0)=x$ to $\sigma_r(a_r)=z$.
Theorem \ref{T:Lipschitz}(iii) asserts $|\dot \sigma_r(0)|_{\tilde g}  \le \tilde C$ for all $x \in X$ and $r \ge R$.
Since $a_r=\ell(x,\gamma(r)) \ge \ell(x,\gamma(R)) + r-R$ diverges as $r \to \infty$,  taking $R$ larger if necessary ensures $a_r \ge 1$ (uniformly with respect to $(x,r) \in X \times [R,\infty)$).

Now $K_{\tilde C} := \{ (v,y) \in TM \mid |v|_g=1, |v|_{\tilde g} \le \tilde C, y \in X\}$ is a compact set of unit timelike directions.
Among timelike geodesics with initial conditions in $K_{\tilde C}$, the time to the timelike cut-locus attains its minimum over
$K_{\tilde C}$ by the lower semicontinuity shown in \cite[Proposition 9.7]{BeemEhrlichEasley96} for 
{the globally hyperbolic} case (a).
This mininum value is positive; call it $t_0 \in (0,\infty]$ and fix $0<t<\min\{t_0,1\}$. The set $G = \{\exp_z sv \mid (v,z,s) \in K_{\tilde C} \times [0,t]\}$ is then compact and contains $\sigma_r({[0, t]})$ for all $x \in X$ and $r \ge R$.

Given $\sigma_r$ as above, fix $y=\sigma_r(t)$.
{Lemma \ref{L:upper support} asserts $u(x') = b^+_{r}(x)+\ell(x,y)-\ell(x',y) \ge b^+_{r}(x')$ holds for all $x'$ in a neighbourhood of $x$ and equality holds at $x'=x$; in other words, $u$ supports $b^+_{r}$ from above at $x$.  
Thus $b^+_{r}$ inherits} from $u$ at $x$ a semiconcavity constant $C(x,y)$ which depends continuously on its arguments in $\{\ell>0\}$
for case (a), according to Proposition~3.4 of \cite{McCann20}.
Taking $C = \max_{(x,y) \in G} C(x,y)$ then concludes the proof of case (a).  

Although the 
proposition last mentioned 
was proved under 
{global hyperbolicity} (a),  the conclusions we need can be extended to {the timelike geodesically complete case} 
(b) by the following observation.
The proof of Proposition~3.4 of \cite{McCann20} provides a semiconcavity constant $C$ of $b^+_r$ at $x=\sigma_r(0)$ given by a simple integral
which depends only on the geometry $(\tilde g,g)$ of $M$ along the proper-time parameterized timelike geodesic $\{\sigma_r(s)\}_{s \in [0,t]}$.  Since this geodesic lies in the compact set $G$,
this explicit formula is easy to bound, allowing us to conclude the same way whenever $t_0>0$.
If $t_0=0$, we 
set $t=1$ and argue as before after observing the resulting set $G$ is again compact since 
the exponential map is continuous and defined globally on the timelike tangent bundle in case (b).
\end{proof}

\begin{corollary}[{Unit gradients converge a.e.\ for Busemann limits}
]\label{C:semiconcave}
If a family $\{b^+_r\}_{r \ge R}$ is equi-semiconcave on an open set $U \subset M$, 
then pointwise convergence to a real-valued limit $b^+=\lim_{r\to \infty} b^+_r$ on $U$
implies locally uniform convergence of $b^+_r$ and pointwise convergence a.e. of $db^+_r$ on $U$.  Moreover, $b^+$ inherits the semiconcavity constants of 
$\{b^+_r\}_{r \ge R}$.  If the family consists of approximate Busemann functions
\eqref{b^+_r} and $M$ is an (a) globally hyperbolic or (b) timelike geodesically complete spacetime, then a.e. on $U$ their gradients are future-directed and $|db^+|=1=|db^+_r|$.
\end{corollary}

\begin{proof}
Semiconcavity implies that in any smooth coordinate chart,  the function becomes a concave function after subtraction of a multiple $C$ of a fixed parabola, depending on the coordinates.
Equi-semiconcavity of $\{b^+_r\}_{r \ge R}$ 
therefore implies locally uniform convergence of $b^+_r \to b^+$ on $U$ and convergence a.e.\ of their gradients,
as for concave functions; cf.\ Theorem 10.9 and the proof of Theorem 25.7 from Rockafellar \cite{Rockafellar70}.
The limit $b^+$ inherits the same constant $C$ of semiconcavity.
Semiconcavity also implies differentiability of $b^+_r$ outside a set of measure zero.
For approximate Busemann functions \eqref{b^+_r}, in the globally hyperbolic case (a)
it is well-known that $|db^\pm_r|=1$ where defined; Theorem 3.6 of~\cite{McCann20}. 
{Under timelike geodesic completeness (b) instead,} this also holds since we already know that $b_r^+$ is a.e.\ differentiable on $U$ and for any $q \in U$ there exists a timelike maximizer to $\gamma(r)$ (for large enough $r$, uniformly on compact subsets of $U$, {by Theorem
\ref{T:Lipschitz}(ii)}).
Thus $|db^+|=1$ a.e.\ on $U$.  The same estimates apply to $-b^-$.
The reverse triangle inequality implies 
\begin{equation}\label{1-steep}
b^\pm_r(y) - b^\pm_r(x) \ge \ell(x,y)
\end{equation}
for all $y \in J^+(x)$,  so $db^\pm_r$ and their limits are both future-directed (recall that the Lorentzian metric has signature $(+,-,\dots,-)$).
\end{proof}

\section{$p$-harmonicity of the Busemann function}

By linearizing the nonlinear inequalities $\square_p b^+ \le 0 \le \square_p b^-$ for $p<1$, in this section we show the sum $b^+-b^- \ge 0$ of the Busemann functions is a weak supersolution to a linear, uniformly elliptic equation.  Since $b^+-b^-$ vanishes on a line,  the strong maximum principle will then yield $b^+=b^-$ in the connected neighbourhood $U$ of this line provided by Theorem \ref{T:Lipschitz}.  Semiconcavity (and $p$-superharmonicity) of $b^+$ combines with semiconvexity (and $p$-subharmonicity) of $b^-$  to imply $b^+=b^- \in C^{1,1}_{loc}(U)$ (and $\square_p b^+=0$ respectively).  

Given a symmetric tensor field $a^{ij}$ 
--- bounded and measurable though not necessarily smooth --- on a Euclidean domain $\Omega \subset \R^n$,  
a Lipschitz function $u$ will be called a {\em weak} solution of $Lu \ge 0$ for the linear operator
\begin{equation}\label{L}
Lu := - \p_j( a^{ij} \p_i u) 
\end{equation}
(or {\em weak supersolution}) if and only if 
\begin{equation}\label{weak supersolution}
\int_\Omega  ( \p_i u) a^{ij} \p_j \phi 
\,dx \ge 0, \quad \forall\ 0 \le \phi \in C^1_0(\Omega);
\end{equation}
here $C^1_0(\Omega)$ denotes the set of continuously differentiable functions with compact support in $\Omega$.
(We prefer to require Lipschitz regularity of $u$ rather than the more customary Sobolev hypothesis $u \in W^{1,2}(\Omega)$.)
Similarly,  $u$ would be called a weak solution of $Lu \le 0$ (or weak {\em subsolution}) if the first inequality in \eqref{weak supersolution} were reversed.
It is called a {\em weak solution} of $Lu=0$ if both inequalities hold.
The operator $L$ is {\em uniformly elliptic} if $a^{ij} 
\in L^\infty(\Omega)$ and satisfy
\begin{equation}\label{elliptic}
a^{ij} v_i v_j \ge \lambda >0  
\end{equation}
for some $\lambda>0$ and all covectors $v \in \R^n$ with Euclidean unit norm.  
For a uniformly elliptic linear operator $L$ on a connected domain $\Omega \subset \R^n$, the remark immediately following
Theorem 8.19 of Gilbarg and Trudinger~\cite{GilbargTrudinger98} asserts that any continuous weak supersolution $u \ge 0$ which vanishes at an interior point
must vanish throughout $\Omega$.  In the proposition below we apply this to the sum $u=b^+-b^-$ of the $p$-superharmonic Busemann functions $b^+$ and $-b^-$
to conclude $b^+=b^-$ in a neighbourhood of the line which defines them.

If instead the coefficients $a^{ij} 
\in L^\infty$ are defined on the cotangent bundle,  hence depend on $du(x)$ as well as $x \in \Omega$,  the operator 
\begin{equation}\label{Q}
Qu := - \p_i( a^{ij} \p_j u) 
\end{equation}
becomes quasilinear.  The same definition \eqref{weak supersolution} of weak supersolution (and the corresponding definitions of weak subsolution and weak solution),
with $Q$ in place of $L$ continue to make sense.  This {turns out to be} the case for the $p$-d'Alembertian $Qu=-\square_pu := \nabla \cdot (|du|^{p-2}du)$, whose weak super and subsolutions
are called {\em $p$-superharmonic} and {\em $p$-subharmonic} respectively.  If $u$ is {simultaneously
$p$-superharmonic and $p$-subharmonic}, i.e. $Qu=0$ weakly, we say $u$ is {\em $p$-harmonic}.

On a spacetime satisying the strong energy condition,  we claim $b^+$ is $p$-superharmonic and $b^-$ is $p$-subharmonic near the line $\gamma$. We only prove the claim for $b^+$, the statement for $b^-$ follows by 
time-reversal symmetry.

Corollary \ref{L:harmonicity} and Proposition \ref{P:maximum principle} below require a weak reformulation which extends the smooth d'Alembert comparison theorem of Eschenburg \cite{Eschenburg88} past the cut locus. 
This is most naturally formulated in terms of our $p$-d'Alembertian; cf.~\eqref{weak supersolution}. 
Such an estimate was recently proved in much greater generality \cite{BeranOctet24+}.  However, for the current smooth setting the following proposition gives a much more direct proof based on the smooth calculation due to Eschenburg \cite{Eschenburg88}.
Integrating his estimate by parts requires some care due to the presence of cut points. On the other hand, by employing the semiconcavity of appropriate Lorentzian distance functions (as in the discussion before 
Proposition \ref{P:semiconcave}), we control the sign of any singular contribution to render it harmless.  Recall $u$ is {\em semiconvex} if $-u$ is semiconcave.

\begin{proposition}[Weak d'Alembert comparison]\label{Pr:dalembert comparison} Fix $0\neq p<1$ and a point $o\in M$. Suppose that there exists an open set  $U \subset I^-(o)$ such that $\ell(\cdot,o)$ is (real and) semiconvex on $U$, the intersection of the past timelike cut locus of $o$ with $U$ is relatively closed with measure zero in $U$, and $\ell(\cdot,o)$ is smooth outside that set {with unit timelike gradient}. Then every nonnegative $\phi\in C_0^1(U)$ satisfies
\begin{align*}
\int_{M} \left[\frac{(n-1)\phi}{\ell(\cdot,o)} + g\Big(d\phi, 
\frac{d\ell(\cdot,o)}{|d\ell(\cdot,o)|^{2-p}}\Big)\right] d\vg \ge 0.
\end{align*}
\end{proposition} 

\begin{proof} We write $u := \ell(\cdot,o)$. By partition of unity, we may and will assume $\phi \in C^1_0(U)$ is supported on a fixed chart. Since $u$ is semiconvex, 
partitioning further if necessary, we may also assume the existence of a smooth function $v$ such that $u+v$ is a convex function of the coordinates in the usual Euclidean sense \cite[Satz 2.3]{Bangert79} . It is also not restrictive to assume that the coordinate representation $(2-p)\,(\p^i u)\,(\p^j u)/\vert du\vert^2 - g^{ij}$ of the tensor from \eqref{nondivergence} is uniformly 
positive-definite on  the support of~$\phi$. By \cite[p.\ 312]{Bangert79}, the constructions below will not depend on the choice of $v$. To relax our notation, we regard $M$ as being covered by a single chart from which it inherits a Euclidean metric.

Standard distribution theory implies the distributional Euclidean Hessian of $u+v$ is given by a 
nonnegative-definite matrix-valued Radon measure $D^2(u+v)$. In turn, as $v$ is smooth the distributional Hessian of $u$ is a signed Radon measure $D^2u$ satisfying
\begin{align*}
    \int_M \partial_i\partial_j\phi\,u\, d x = \int_M \phi\,d(D^2u)_{ij}.
\end{align*}
Moreover, the $dx$-singular part $(D^2 u)^{\perp}$ of $D^2u$ is nonnegative-definite.

Given $\varepsilon > 0$ let $\rho_\varepsilon$ be a standard convolution kernel, and set $u_\varepsilon := u * \rho_\varepsilon$. By the Lipschitz regularity assumed of $u$ on $U$, for sufficiently small $\varepsilon > 0$ the matrix $\smash{(2-p)\,
{\partial_{i} u_\varepsilon\partial_{j} u_\varepsilon}
/\vert du_\varepsilon\vert^2 - g_{ij}}$ stays uniformly positive definite and $\vert d u_\varepsilon\vert > 0$ on the support of $\phi$. Recalling \eqref{nondivergence}, the $p$-d'Alembertian of $u_\varepsilon$ is
\begin{align}\label{Eq:dAlem}
    \square_pu_\varepsilon = \vert du_\varepsilon\vert^{p-2}\left[(2-p)g^{ik}g^{jl}\frac{\partial_{k} u_\varepsilon\partial_{l} u_\varepsilon}{\vert du_\varepsilon\vert^2}-g^{ij}\right]\left[\partial_i\partial_ju_\varepsilon - \Gamma_{ij}^k\,\partial_ku_\varepsilon\right].
\end{align}
Since $u_\varepsilon$ is smooth, we can transform the right-hand side into its divergence form and use integration by parts to obtain
\begin{align}\label{Eq:IBPf}
     \int_M (g^{ij}\,\partial_i\phi\,\frac{\partial_i u_\varepsilon}{\vert du_\varepsilon\vert^{2-p}})\,\sqrt{\vert g\vert}\,d x = \int_M (\phi\,\square_pu_\varepsilon)\,\sqrt{\vert g\vert}\,dx.
\end{align}
On the other hand, convolution commutes with differentiation; in conjunction with the nonnegativity of $(D^2u)^\perp$ this entails
\begin{align*}
    \partial_i\partial_ju_\varepsilon = \rho_\varepsilon * (D^2u)_{ij} \geq \rho_\varepsilon * (D^2 u)_{ij}^{\mathrm{a.c.}}. 
\end{align*}
Here $(D^2 u)^{\mathrm{a.c.}}$ is the $dx$-absolutely continuous part of $D^2u$, and the last inequality is understood in the sense of 
positive-semidefiniteness of matrix-valued distributions. Since the Hessian $\partial_i\partial_j u_\varepsilon$ is multiplied by a symmetric and 
positive-definite matrix in \eqref{Eq:dAlem} which admits a square-root, the previous observation combines with \eqref{Eq:IBPf} and the nonnegativity of $\phi$ to give
\begin{align*}
   \int_M (g^{ij}\, \partial_i\phi\,\frac{\partial_i u_\varepsilon}{\vert du_\varepsilon\vert^{2-p}})\,\sqrt{\vert g\vert}\,d x \geq \int_V (\phi\,\square_pu_\varepsilon)\,\sqrt{\vert g\vert}\,dx.
\end{align*}
Here $V$ is the open subset of $U$ where $u$ is smooth. Its complement relative to $U$ is negligible 
since we assumed the past timelike cut locus of $o$ has measure zero in $U$.

As $\varepsilon \to 0$, the left-hand side converges to
\begin{align*}
    \int_M g^{ij}\,\partial_i\phi\,\frac{\partial_i u}{\vert du\vert^{2-p}}\,\sqrt{\vert g\vert}\,d x 
\end{align*}
by the dominated convergence theorem, since $\phi \in C^1_0(U).$ 
By the same argument, the term containing the Christoffel symbols in \eqref{Eq:dAlem} in the integral on the right-hand side converges. Lastly, since $u$ differs from a convex function by the addition of a smooth function $v$ whose Hessian is bounded on the support of $\phi$, Fatou's lemma applies and yields the following limiting estimate for the right-hand side of the previous inequality:
\begin{align*}
\liminf_{\varepsilon \to 0} \int_V (\phi\,\square_pu_\varepsilon)\,\sqrt{\vert g\vert}\,dx \geq\int_V(\phi\,\square_pu)\,\sqrt{\vert g\vert}\,dx.
\end{align*}

Finally, by Eschenburg's smooth d'Alembert comparison \cite[§5]{Eschenburg88} (note the different sign conventions on the metric tensor) and the definition of $u$, on the given set $V$ its $p$-d'Alembertian is bounded from above by $(n-1)/\ell(\cdot,o)$. This implies
\begin{align*}
 \int_M g^{ij}\,\partial_i\phi\,\frac{\partial_i u_\varepsilon}{\vert du_\varepsilon\vert^{2-p}}\,\sqrt{\vert g\vert}\,d x  \geq -\int_M \frac{\phi(n-1)}{l(\cdot,o)}\,\sqrt{\vert g\vert}\,dx,
\end{align*}
as desired.
\end{proof}

\begin{corollary}[Busemann function $b^+$ is $p$-superharmonic]
\label{L:harmonicity}
Fix $0 \ne p<1$. 
Under the hypotheses of Theorem \ref{T:Lorentzian splitting}(a) or (b) on $(M,g)$ and $\gamma$,
if $V \subset M$ is a domain on which the functions $\{b^+_r\}_{r \ge R}$ of \eqref{b^+_r} are equi-semiconcave when $R$ is sufficiently large, and if in addition 
for all $p \in V$ and $r \ge R$ there exists a timelike maximizing geodesic from $p$ to $\gamma(r)$,
then $b^+ = \lim_{r\to \infty} b^+_r$ 
is semiconcave, $p$-superharmonic and $|db^+|=1$ on $V$.
\end{corollary}

\begin{proof} The reverse triangle inequality \eqref{rti} shows
$I^-(\gamma(r))$ increases with $r$;  since $b^+_r(x) = +\infty$ unless $x \in I^-(\gamma(r))$, we have $V \subset \cup_{r>0} I^-(\gamma(r))$.
Using normal coordinates around $\gamma(r)$
shows $b^+_r(\cdot)=const(r)-\ell(\cdot,\gamma(r))$ to be smooth on $V$, except on the closure of the timelike cutlocus.

{Under global hyperbolicity} (a),  the timelike cutlocus of 
$\gamma(r)$ intersects $I^-(\gamma(r))$ in a 
relatively closed set \cite[Lemma 2.3]{McCann20} 
which has zero volume as a consequence of Theorem 3.5
and the 
approximate second-differentiability a.e.\ of the semiconvex function $ \ell(\cdot, \gamma(r)) = const(r)-b_r^+(\cdot)$ 
described following Definition 3.8 (ibid).
Proposition \ref{Pr:dalembert comparison} then applies to the approximate Busemann function $b_r^+$ for fixed $r>0$ and yields
\begin{equation}\label{weak p-comparison}
\int_{M} \left[\frac{(n-1)\phi}{\ell(\cdot,\gamma(r))} - g\Big(d\phi, 
\frac{db^+_r}{|db^+_r|^{2-p}}\Big)\right] d\vg \ge 0
\end{equation}
for  every nonnegative $\phi \in C^1_0(I^-(\gamma(r)))$. 
The convergence provided by Corollary \ref{C:semiconcave} allows us to take $r \to \infty$ in \eqref{weak p-comparison} to get $p$-superharmonicity of $b^+$, semiconcavity and $|db^+|=1$ on any open subset $X$ with compact closure in $V$ by Lebesgue's dominated convergence theorem, where the monotone limit $\lim_{r\to \infty}\ell(\cdot,\gamma(r))= \infty$ and the compact support of $\phi \in C^1_0(X)$ have been used. 
This concludes case (a).

{The timelike geodesically complete} case (b) will follow similarly once we have verified that
the timelike cutlocus of $\gamma(r)$ intersects $I^-(\gamma(r))$
in a relatively closed set of zero volume for each $r \ge R$.
Although Theorem 3.5 (ibid) is stated under hypothesis (a),  inspection
of its proof reveals it applies equally well to (b) as soon as $r \ge R$.
This is not true of Lemma 2.3 (ibid) however, so we must find another argument to show the intersection in question is relatively closed.
Recall the (past) timelike cut locus is contained in the graph of 
$G(v) := \exp s(v)v \in M$ over the past unit observer bundle
$T^-_{1}M := \{v \in TM \mid H(-v)=1 \}$, where $s:T^-_{1}M \longrightarrow [0,\infty]$ is upper semicontinuous according to
\cite[Proposition 9.5]{BeemEhrlichEasley96} in case (b).  We shall complete the proof by arguing that $s$ is also lower semicontinuous on the intersection of $T^-_1 M$ with $\exp_{\gamma(r)}^{-1} V \subset T_{\gamma(r)} M$ provided $r \ge R$,  hence $G$ is continuous there.
In case (a) this would follow from \cite[Proposition 9.7]{BeemEhrlichEasley96}. However, the proof of that proposition reveals global hyperbolicity is used only to guarantee the existence of a maximizing geodesic linking $\gamma(r)$ to each $x \in V$,  which the present context is guaranteed by hypothesis. Case (b) is therefore resolved.
\end{proof}

\begin{proposition}[Strong tangency principle]\label{P:maximum principle}
If the spacetime $(M,g)$ of Proposition \ref{P:semiconcave} satisfies the strong energy condition, then 
$b^\pm = \lim_{r\to \infty} b^\pm_r$ from \eqref{b^+_r}--\eqref{b^-_r}
satisfy $b^+=b^- \in C^{1,1}_{loc}(U)$ and $|db^\pm |=1$ and are $p$-harmonic for all $0\ne p<1$ on a 
neighbourhood $U$ of the line $\gamma$.  
\end{proposition}

\begin{proof}
Proposition \ref{P:semiconcave},
Corollary~\ref{L:harmonicity} and Theorem \ref{T:Lipschitz}(ii)
combine with Corollary \ref{C:semiconcave} and time-reversal symmetry to show 
$-b^-$ and $b^+$ are $p$-superharmonic, semiconcave, and have unit timelike gradients, past-directed in case of $-b^-$ and future-directed in case of $b^+$.
Thus  $u:=b^+- b^-$ and $b(t) := b^- + tu$ with $0 \le \phi \in C^1_0(U)$ yield
\begin{align*}
0 & \le- \int_M d\vg \int_0^1   \frac{d}{dt} g(d\phi, |db|^{p-2} db) dt
\\ &= -\int_U d\vg \int_0^1  |db|^{p- 2} g(d\phi, [(p-2)\frac{db}{|db|} \otimes \frac{\nabla b}{|\nabla b|}+I] d u) dt
\\ &= +\int_\Omega dx \sqrt{|g|} \p_i\phi \p_j u
\int_0^1  |db|^{p- 2}\left[
(2-p) \frac{\p^i b \p^jb}{|d b|^2} - g^{ij} \right] dt
\end{align*} 
where the last line is expressed in a coordinate chart;
we assume $\spt \phi$ is supported in such a chart without loss of generality. Viewing the coefficients as `frozen' shows $u$ is a weak supersolution $Lu\ge0$ of the {\em linear} operator given in divergence form \eqref{L}--\eqref{weak supersolution} by
\begin{align}
\label{A matrix}
a^{ij}(x) &= {\sqrt{|g|}}\int_0^1  |db|^{p- 2}\left[
(2-p) \frac{\p^i b \p^jb}{|d b|^2} - g^{ij} \right] dt 
\end{align}
Choosing Fermi coordinates near the line $\gamma$,
 our signature $(+, -, \dots, -)$ of $g$ with $0 \ne p<1$
 make it easy to check bounded measurability and uniform ellipticity \eqref{elliptic} of these coefficients,
taking $U$ smaller if necessary: along the line $\gamma(r)$ where $db^\pm =d\gamma/dr$,  the expression in square brackets becomes the diagonal matrix $diag(1-p,1,\ldots,1)$;
near~$\gamma$, semiconcavity ensures that the gradients $db^\pm$ are not very different from $d\gamma/dr$ either, Theorems 24.4 and 25.1  of \cite{Rockafellar70}.
 On the other hand,  $u \ge 0$ throughout $U$ and vanishes on $\gamma$, according to \eqref{ordering}.
Thus $u=0$ throughout a neighbourhood $U$ of $\gamma$ by the strong maximum principle, Theorem~8.19 of~\cite{GilbargTrudinger98}.
Now $b^+=b^-$ is both semiconvex and semiconcave, hence $b^+ \in C^{1,1}_{loc}(U)$; similarly $b^+=b^-$ is both $p$-super and $p$-subharmonic,  hence $p$-harmonic.
\end{proof}

\begin{remark}[Ellipticity]
Although it is easier to verify the ellipticity claimed using semiconcavity of $\pm b^\pm$,  it can alternatively be deduced from Theorem \ref{T:Lipschitz} using Remark \ref{R:conic intersections}, which ensures the Clarke subdifferential of $b(t)$ remains bounded at each point on $\gamma$.   Here the Clarke subdifferential~\cite{Clarke83} of $b(t)$ at $x$ refers to the closed convex hull of limits of $db(t)(x_k)$ along sequences $x_k \to x$ of points in $M$ of differentiability of $b(t)$.
\end{remark}

\begin{remark}[Higher regularity]
Although it might be possible to obtain higher regularity for $p$-harmonic $C^{1,1}_{loc}$ functions using 
Evans's \cite{Evans82} or Krylov's \cite{Krylov82} techniques, 
it does not follow from their stated results and we were not successful in adapting their methods.  Instead we shall establish it in the next section using a radically simpler approach.
\end{remark}

\section{A Bochner-Ohta identity of homogeneity $2p-2<0$}

In this section we establish a nonlinear
Bochner-Ohta formula on Lorentzian spacetimes;
this will eventually imply that the Levi-Civita Hessian of our $p$-harmonic Busemann function $b^+$ vanishes.
After discovering this formula,  we realized that 
a large family of similar identities were
previously established in Theorem 4.4 of Ohta \cite{Ohta14h}
for Hamiltonians which are smooth and strongly convex away from the zero section of a manifold admitting a Riemannian structure.
Although our result is similar in spirit, 
it complements them in the sense that our Hamiltonian, being adapted to the Lorentzian setting,
satisfies neither the smoothness nor uniform convexity stipulated there.
In particular, our more specialized setting allows us to give a simple statement and proof in terms of standard differential geometric concepts; cf.~Remark~\ref{R:simpler but longer}.

On our $n$-dimensional, signature $(+, -, \cdots, -)$ spacetime $(M,g)$,  the Levi-Civita connection is denoted by $\nabla$. 
Equip the cotangent bundle with a Hamiltonian $H(v,x) = f(g(v,v)/2)$
on timelike 
future co-vectors, where it depends smoothly on the Lorentz norm of $v$, and is irrelevant elsewhere.  
Denote derivatives of $H$ with respect to $v=(v_1,\ldots,v_n)$ by $DH \in TM$, having components $H^i = \frac{\pt H}{\pt v_i}$,
and the Hessian $D^2H=(H^{ij})_{1\le i,j \le n}$ and higher $v$ derivatives of $H$ similarly.
For a function $u \in C^3(M)$
whose gradient is future-directed and timelike everywhere,  
the identity we derive in  \eqref{Bochner2} is the following:
\begin{align}
&\nabla \cdot [(D^2 H|_{\d u}) \d (H|_{\d u})] -  
(DH) \d [ \nabla \cdot (DH|_{\d u})]
\nonumber \\ \label{Bochner} 
&= \mathrm{Tr}[(D^2H)(\nabla^{2}u)(D^2H) (\nabla^2 u)] + \Ric(DH,DH), 
\end{align}
where $H$ and its derivatives are all tacitly composed with $\d u$,
adjacent tensors are contracted in the obvious way (see \eqref{Bochner2}), and $\nabla^2 u=(\nabla_i\nabla_j u)_{1\le i,j\le n}$ denotes the Levi-Civita Hessian of $u$. 
(We use $\d$ here simply to denote the differential of a function. On an orientable spacetime, we can also use the adjoint operator $\d^*$ to denote divergence instead of $\nabla \cdot$.)

The relevance of this identity is the following. Choose $H(v,x)$ from \eqref{Hamiltonian}
with $0\neq p<1$, in which case the identity we establish has homogeneity $2p-2<0$.
If  $u$ is such that $du$ has constant Lorentz norm and $\square_p u = 0$, the first two terms vanish, since $\nabla \cdot (  DH|_{du})=\square_p u=0$.
The strong energy condition \eqref{strong energy} 
combined with convexity of $H$ on timelike future covectors $v$ makes the right hand side of \eqref{Bochner} strictly positive unless the Hessian of $u$ vanishes identically (and the timelike Ricci curvature vanishes in direction $DH$ or equivalently, $\d u$).  
Applied to the $p$-harmonic function $u=b^+$,  this identity yields the desired linearity and smoothness of the Busemann function.

\begin{lemma}[A Lorentzian Bochner-Ohta identity] \label{L:Bochner-Ohta}\ \\
If $H(v,x) = f(v_iv_j g^{ij}(x)/2)$
on the timelike future bundle of covectors to a spacetime $(M,g)$ with $f \in C^{3}((0,\infty))$, 
and the differential of $u \in C^3(M)$ is future-directed and timelike everywhere,  then evaluating $H$ and all its derivatives at 
$du$ 
yields 
\begin{align}
\label{Bochner2} & 
\nabla_i (H^{ij}|_{\d u} \nabla_j (H|_{\d u})) - H^i \nabla_i (\nabla_j (H^j |_{\d u})) 
=H^{ij}u_{jk} H^{kl}u_{li} + R_{ij}H^iH^j, 
\end{align}
where superscripts denote derivatives of $H(v,x)$ with respect to components of the covector $v=(v_1,\ldots,v_n)$,
subscripts denote covariant derivatives with respect to the Levi-Civita connection, except $R_{ij}$ is the Ricci curvature tensor,
and the Einstein summation convention holds. 
\end{lemma}

\begin{proof}
Evaluating  the left hand side of \eqref{Bochner2} using the chain rule (and the fact that $v$ derivatives of $H$ all commute with each other
since the cotangent space at each point is flat,  while $x$ derivatives of $H$ vanish due to the form of our Hamiltonian) yields
\begin{align*}
&\nabla_i(H^{ij}H^ku_{jk}) - H^k \nabla_k(H^{ij} u_{ji})
\\=&H^{ijl}H^k(u_{il}u_{jk}-u_{kl}u_{ji}) + H^{ij}u_{jk} H^{kl} u_{il} +H^{ij} H^k (u_{ijk}-u_{kji}) 
\end{align*}
Here $u_{ijk} := \nabla_i \nabla_j \nabla_k u$.
The terms involving third derivatives of $H$ cancel each other (since superscripts on $H$ can be freely permuted and the Levi-Civita connection is torsion free).  It remains to see only that the terms involving
third derivatives of $u$ constitute the Ricci curvature term on the right hand side of \eqref{Bochner2}.  But this follows by combining consequences
\begin{align}
\label{h' identity}
H^i &= (\nabla^i u) f'_{|\d u|^2/2} \qquad {\rm and} \qquad
\\H^{ij} &= (\nabla^i u) (\nabla^j u) 
f''_{|\d u|^2/2} + g^{ij} f'_{|\d u|^2/2}
\label{h'' identity}
\end{align}
of the structure of our Hamiltonian with the defining property (and antisymmetry) of the Riemann tensor
\begin{align*}
 u_{ikj}-u_{kij} = {R_{ikj}}^l u_l. 
\end{align*}
\end{proof}

\begin{remark}[A simpler but longer formula]\label{R:simpler but longer}
Formula \eqref{Bochner2} can also be written in terms of $f$ instead of $H$ using \eqref{h' identity}--\eqref{h'' identity} in which superscripts on the right hand side now correspond to standard tensor indices raised using the Lorentzian metric tensor,
an even more substantial simplification relative to the analogous expressions in \cite{Ohta14h}.
\end{remark}

The next corollary applies our result to the power-law Hamiltonian \eqref{Hamiltonian};  the formula \eqref{Bochner3} it contains also seems simpler to us than the variant developed in \cite[Appendix]{MondinoSuhr23}.

\begin{corollary}[Linearity and smoothness of Busemann functions]\label{C:linear}
From the  strong energy condition and the conclusions of Proposition~\ref{P:maximum principle} it follows 
for $g_{ij} \in C^{\infty}(M)$ that
$b^+ \in C^{\infty}(U)$ 
and its Hessian vanishes $\nabla_i \nabla_j b^+=0$ 
throughout $U$. 
\end{corollary}

\begin{proof}
Specializing \eqref{Bochner2} to the Hamiltonian \eqref{Hamiltonian} with $0\ne p<1$, under    the strong energy condition \eqref{strong energy},
Lemma~\ref{L:Bochner-Ohta} yields
\begin{align}\nonumber
& g(|du|^{p-2} du , d (\square_p u)) 
-\nabla \cdot ((D^2 H) d (|du|^p/p)) 
\\ \label{Bochner3} 
&=\mathrm{Tr}\left[ \sqrt{D^2H} (\nabla^2 u) D^2 H (\nabla^2 u) \sqrt{D^2H}\right] + |du|^{2p-4} R(du, du),
\\& \ge \mathrm{Tr}\left[ \sqrt{D^2H} (\nabla^2 u) (D^2 H) (\nabla^2 u) \sqrt{D^2H}\right] 
\nonumber
\end{align}
where we have used convexity of the Hamiltonian to take the matrix square-root: strict positive-definiteness of $D^2 H$ on the timelike future cone was shown in Lemma 3.1 of McCann
\cite{McCann20}.
Proposition \ref{P:maximum principle} 
provides a neighbourhood $U$ of the line $\gamma$ such that $b^+ \in C^{1,1}_{loc}(U)$ satisfies $|db^+|=1$ and 
$\square_p b^+=0$ a.e. on $U$.
If $b^+ \in C^3(U)$,  the left-hand side of \eqref{Bochner3} vanishes when $u=b^+$; in this case we conclude $ \sqrt{D^2H} (\nabla^2 b^+) \sqrt{D^2H}=0$ throughout $U$, and
positive-definiteness of $D^2 H$ yields the desired linearity $\nabla_i \nabla_j b^+=0$ of~$b^+$.

If instead $b^+ \in C^{1,1}_{loc}(U)=W^{2,\infty}_{loc}(U)$,  there exists a 
sequence $u_\ep \in C^\infty(U)$ with $|du_\ep| \ge 1-\ep$ and $\|u_\ep-b^+\|_{W^{2,r}(X)} \to 0$ for each $r\in[1,\infty)$ and compact $X \subset U$;  moreover $\nabla^2 u_\ep \to \nabla^2 b^+$ (hence
$\square_p u_\ep \to \square_p b^+$) 
$\vg$-a.e.\ \cite[\S 5.3.1 and C.4]{Evans98}.
For each test function $0 \le \phi \in C^1_0(U)$,  
evaluating \eqref{Bochner3}
at $u_\ep$ and integrating against $\phi$
yields
\begin{align*}
&\int_U [ (\square_p u_\ep)\nabla \cdot (\phi DH_\ep)
- D^2 H_\ep (d\phi,d(H_\ep)) ] d\vg
\\ &\ge 
\int_U \phi \mathrm{Tr}\left[ \sqrt{D^2H_\ep} (\nabla^2 u_\ep) (D^2 H_\ep) (\nabla^2 u_\ep) \sqrt{D^2H_\ep}\right] d\vg
\end{align*}
where $D^2H_\ep:=D^2 H|_{du_\ep}$ and similarly $DH_\ep= DH|_{du_\ep}$ and $H_\ep=H|_{du_\ep}$.
Setting $u_0 := b^+$ and $X =\spt \phi$ yields
$D^2 H_\ep \to D^2 H_0$ and $\nabla \cdot (\phi DH_\epsilon) \to \nabla \cdot (\phi DH_0)$ in $L^{r}(X)$ 
for all $r \in [1,\infty)$.  
Since the sequences
$d H_\ep \to d H_0=0$ and $\square_p u_\ep \to \square_p u_0=0$
and $\nabla^2 u_\ep \to \nabla^2 u_0$ converge
$\vg$-a.e. on $X$ and are bounded, they also converge in $L^r(X)$
for all $r \in [1,\infty)$. Choosing $r=2$,
the $\ep \to 0$ limit yields
\begin{align*}
0 \ge 
\int_U \phi \mathrm{Tr}\left[ \sqrt{D^2H_0} (\nabla^2 u_0) (D^2 H_0) (\nabla^2 u_0) \sqrt{D^2H_0}\right] d\vg.
\end{align*}
Since $b^+ \in C^{1,1}_{loc}(U)$,  the covariant Hessian 
$\nabla^2 b^+$ is absolutely continuous with respect to $\vg$;
positive-definiteness of $D^2H_0$ implies $\nabla^2 b^+$ vanishes $\vg$-a.e. --- hence everywhere --- on $X$.  
Arbitrariness of 
$0 \le \phi \in C^1_0(U)$ concludes the proof that $b^+$
is smooth and linear throughout $U$.
\end{proof}

\section{Local splitting in Newman's setting}

Having established linearity of the Busemann function $b^+$ in a neighbourhood of the line $\gamma$, we can prove a local version of the splitting (Theorem \ref{T:Lorentzian splitting}).
Although our strategy is inspired by 
that of Eschenburg~\cite{Eschenburg88},
Galloway \cite{Galloway89}, and Galloway and Horta \cite{GallowayHorta96},  it is {much simpler than} 
these (as well as the textbook 
proof~\cite{BeemEhrlichEasley96}),  due to the fact we already know that the Busemann functions $b^+=b^-$ 
are linear in an entire neighbourhood $U$ of $\gamma$,  and not merely on the intersection of this neighbourhood with some well-chosen spacelike surface such as the zero level set $S_0$, where $S_r = \{x \in M \mid b^+(x)=r\}$ for each $r \in \R$. 
The vanishing Hessian of $b^+$ already shows $S_0$ to be totally geodesic in $U$,  so we need no recourse to Bartnik's existence result for surfaces of zero mean curvature \cite{Bartnik88b} either.
However,  we shall still need to show the vanishing of this Hessian propagates to all asymptotes of $\gamma$ that pass through $S_0$ in this neighbourhood.

To define asymptotes, we recall the terminology of Galloway and Horta~\cite{GallowayHorta96}:
Let $\LL(\sigma)$ denote the {Lorentzian} length (i.e.\ proper-time) along a future-directed (hence causal) curve $\sigma:[s,t]\longrightarrow M$ 
from \eqref{action} above.
Given $t \in (0,\infty]$, a set $S\subset M$,  a ray $\sigma:[0,t) \longrightarrow M$ is called an {\em $S$-ray} if $\gamma$ maximizes distance to $S$, i.e. if 
$$
\LL(\sigma|_{[0,s]})=\ell(S,\sigma(s)) \qquad \forall 0 \le s <t,
$$
where $\ell(S,x) := \sup\{\ell(y,x) \mid y \in S\}$. Thus a ray $\sigma$ is also a $\{\sigma(0)\}$-ray.
A sequence of future-directed curves $\sigma_k:[s_k,t_k]\longrightarrow M$ is called {\em limit maximizing} if 
$$
0 \le \liminf_{k \to \infty} [\LL(\sigma_k) - \ell(\sigma_k(s_k),\sigma_k(t_k))].
$$
Given a complete $S$-ray $\gamma$,  
a {\em generalized co-ray}
refers to a ray constructed as a 
limit curve $\sigma:[0,t)\longrightarrow M$ 
--- in the sense of \cite{GallowayHorta96} ---  
of a limit maximizing sequence
$\sigma_k:[0,s_k] \longrightarrow M$ with $\lim_{k \to\infty}  \sigma_k(0)=\sigma(0) \in I^-(\gamma) \cap I^+(S)$ and $\sigma_k(s_k)=\gamma(r_k)$ and $r_k \to \infty$.
If  $\LL(\sigma_k) =  \ell(\sigma_k(s_k),\sigma_k(t_k))$, meaning the curves $\sigma_k$ are all maximizing, then $\sigma$ is called a {\em co-ray}.
If $\sigma_k(0)=\sigma(0)$ for each $k$, the co-ray is called an {\em asymptote}.

We are now ready to give a simple proof of the following local splitting theorem, proved for $g_{ij} \in C^\infty(M)$
by Eschenburg \cite[Proposition~6.3]{Eschenburg88}  under hypothesis (a) plus (b), Galloway \cite{Galloway89} under {global hyperbolicity} (a) and by Newman \cite{Newman90}
(also Galloway and Horta \cite[\S 5]{GallowayHorta96}) 
under {timelike geodesic completeness} (b).
The initial argument covers both cases
but then bifurcates:  we complete the proof of Newman's case (b) in the present section and defer the completion of Galloway's case (a) to the following section.

\begin{theorem}[Local splitting]\label{T:local splitting}
Under the hypotheses of Theorem \ref{T:Lorentzian splitting}(a) or (b),
there is a neighbourhood $W \subset M$ of $\gamma$ which splits:  
There is a 
smooth spacelike hypersurface $S \subset S_0 \subset M$ containing $\gamma(0)$ and a diffeomorphism $E:\R \times S \longrightarrow W$ given by 
$E(r,x)= \exp_x r db^+$ which is a local isometry in the sense that $E$ pulls the metric $g$ back to the product metric $dr^2 - h$,  where $-h$ is the restriction of $g$ to $S$.
Moreover, $r \in \R \mapsto E(r,x)$ is a line maximizing time to $S_0:=\{b^+=0\}$ for each $x \in S$, and $\gamma(\cdot) = E(\cdot,\gamma(0))$. 
\end{theorem}

\begin{proof}[Proof in {(b) timelike geodesically complete} case]
Taking the neighbourhood $U$ of $\gamma$ from Corollary \ref{C:linear} --- on which the Hessian of $b^+ \in C^{\infty}(U)$ vanishes --- 
smaller if necessary,
ensures each $x \in S:=S_0 \cap U$ forms the 
base $x=\sigma^\pm(0)$ of forward and backward asymptotes $\sigma^+:[0,s_+)\longrightarrow M$
and $\sigma^-:(s_-,0] \longrightarrow M$ to $\gamma$ (with $s_+=\infty=-s_-$ in case (b)),
both proper-time reparameterized to be future-directed;
 this follows by using Lemmas 2.1--2.4 of Galloway and Horta~\cite{GallowayHorta96} to extract timelike subsequential limits $r\to \infty$ of the maximizing segments from $x$ to $\gamma(\pm r)$ provided for $r$ sufficiently large by our Theorem \ref{T:Lipschitz}(ii)-(iii).
Together they form a future-directed geodesic $\sigma$, potentially broken at $\sigma(0)\in S$.  Proposition~2.6 of the same reference shows 
$b^+(\sigma(r))=r$ for all $r \in [0,s_+)$;
then
\begin{align*}
b^-(\sigma(r)) 
&= b^-(\sigma(r)) - b^-(\sigma(0))
\\& \ge r = b^+(\sigma(r))
\end{align*}
couples with \eqref{ordering} $b^+ \ge b^-$ 
to yield $b^-(\sigma(r))=r$ in the same range 
of $r \in [0,s_+)$.
Combining the preceding argument with time-reversal symmetry shows both Busemann functions
increase along the full asymptote at unit rate: $b^\pm(\sigma(r))=r$ for all $r \in (s_-,s_+)$. 
Because
$$
b^+(y)-b^+(x) \ge \ell(x,y)
$$
for all $x \ll y$,  with equality at $(x,y)=(\sigma(r),\sigma(s))$ for all $r<s\in (s_-,s_+)$,  it follows that (i) that $\sigma$ is a timelike maximizer;
(ii) its tangent $\sigma'(r)=N(\sigma(r))$ agrees with the 
direction $N=db^+$ of slowest increase of $b^+$ on $U$;
and (iii) $\sigma^\pm$ are $S_0$-rays.
 Thus $\sigma$ is not broken at $\sigma(0)$ after all.  Moreover,  $\sigma$
is normal to the zero level set  $S \subset S_0$ of $b^+$, which is spacelike since $N$ is unit timelike, and smooth by the implicit function theorem.

The arguments above show that $E(r,x):=\exp_x rN$, with maximal domain of definition $\mathcal{D} \subset \R \times S$, has the following properties: The trajectories $E(\cdot,x)$ and $E(\cdot,y)$, for $x,y \in S$, are timelike lines which do not cross (unless $x=y$), contain no conjugate points nor focal points to $S$, and maximize the time separation to $S$.  In {the globally hyperbolic} case (a) $\mathcal{D}$ is open
by \cite[Proposition 9.7]{BeemEhrlichEasley96}, while in {the timelike geodesically complete} case
(b), $\mathcal{D} = \R \times S$ holds automatically.
In both cases the implicit function theorem shows $E$ is a diffeomorphism from $ \mathcal{D} \subset \R \times S$ onto a set $W \subset M$ which is open, 
and $E(r,\gamma(0))=\gamma(r)$ for all $r \in \R$.

We have now shown $b^+=b^-$ throughout $W$ (in addition to $U$).
Since $E$ is a smooth diffeomorphism and $b^+(E(r,x))=r$ for each $(r,x)\in \mathcal{D}$,  the continuous differentiability of $b^\gamma:=b^\pm$ 
propagates throughout $W$. On $U$,  recall $b^\gamma$
is smooth and its Hessian vanishes, from Corollary~\ref{C:linear}.
Thus $N=db^\gamma$ satisfies Killing's equation in $U$,  hence is parallel throughout $U$, so that any geodesic in $U \subset M$ has constant inner product with $N$.
In particular, since $N=db^\gamma$ is normal
to the zero level $S^\gamma := S_0\cap U =S_0\cap W$ of $b^\gamma$,
we see $S$ is totally geodesic: its second fundamental form (or shape operator) vanishes. 

{\bf At this point we specialize to {(b) the timelike geodesically complete} case,}
defering the completion of {(a) the globally hyperbolic} case to the following section.
In case (b) the line $\sigma$ is complete,
so applying the same logic to it as to $\gamma$
yields a neighbourhood $V$ of $\sigma$ on which the forward Busemann function
$b^\sigma$ is smooth and has vanishing Hessian.  Thus its zero set $S^\sigma := \{x\in V \mid b^\sigma(x)=0 \}$ is also totally geodesic in $V$.
Taking $V$ smaller if necessary again ensures each $x \in S^\sigma$ lies on a 
complete line $\beta^x:(-\infty,\infty) \longrightarrow M$ 
with $\beta^x(0)=x$
which is (future and past) asymptotic to $\sigma$.
Since both $S^\gamma$ and $S^\sigma$ are orthogonal to $\sigma$ and are totally geodesic,  we conclude they coincide in the 
neighbourhood $V \cap W$ of $\sigma(0)$.  Now the asymptotes to $\sigma$
and to $\gamma$ both intersect $S$ orthogonally in $W \cap V$, so for each
$x \in S\cap V$ these asymptotes ($E(\cdot,x)$ to $\gamma$ and $\beta^x(\cdot)$ to $\sigma$) also coincide.  Moreover both $b^\gamma$ and $b^\sigma$ increase at rate one along these asymptotes,  hence $b^\gamma=b^\sigma$ throughout $W\cap V$.  Thus $b^\gamma$ inherits the linearity of $b^\sigma$: its Hessian
vanishes throughout $W \cap V$.  Since the asymptote $\sigma$ of $\gamma$ had arbitrary base $\sigma(0) \in S$,  we conclude the Hessian of $b^\gamma$ vanishes globally on $W$.

We have now shown the flow map $\Phi(r,x)$ 
of $N=db^\gamma$ satisfying $\frac{d\Phi}{dr} = N(\Phi(r,x))$ and $\Phi(0,x)=x$ 
coincides with $E$ on $\R \times S$.
The fact that the Hessian of $b^\gamma$ vanishes  means $N=db^\gamma$ is parallel throughout $W$ and satisfies Killing's equation.
Thus $\Phi(r,\cdot):W \longrightarrow W$ pulls-back the metric $g$ 
to itself for each $r \in \R$, which shows $E(r,S)$ to be isometric to $S=E(0,S)$.
Along the totally geodesic surface $S$ normal to $N$,  the metric $g$ therefore splits into the direct sum of its restriction $-h$ to $S$ plus $dr^2$ in the orthogonal direction $N$.  Thus $E$
gives the desired local isometry between 
$(\mathcal{D}=\R\times S, dr^2-h)$ and $(W,g)$.
\end{proof}

\begin{corollary}[Local to global isometry]\label{C:local splitting}
With the hypotheses and notation of Theorem~\ref{T:local splitting}, the time-separation function $\ell$ on $M^2$ satisfies
\begin{equation}\label{local to global}
\ell(E(s,x),E(t,y)) \ge 
\begin{cases}
\sqrt{(t-s)^2 - d^2_h(x,y)} & {\rm if}\ t-s \ge d_h(x,y),
\\ -\infty & {\rm else},
\end{cases}
\end{equation}
where $s,t \in\R$ and $d_h$ denotes the Riemannian distance in $S$ between $x,y \in S$.  Moreover, if $W=M$ then this estimate becomes an equality so the isometry becomes global.
\end{corollary}

\begin{proof}
This follows from Theorem \ref{T:local splitting}, the definition of the product metric $dr^2-h$, and of $\ell$ from \eqref{global time separation}.
\end{proof}

\section{Local splitting in Galloway's setting}

To adapt the local splitting from the 
(b) timelike geodesically complete setting of Newman \cite{Newman90}
to the (a) globally hyperbolic setting of Galloway \cite{Galloway89} requires additional arguments
to rule out any possible incompleteness of the inextendible
asymptotes constructed in its proof.  Readers interested only
in timelike geodesically complete spacetimes (b) can skip this section.  We begin with a criterion for asymptotes to be timelike, which can be viewed as a partial converse to Theorem \ref{T:Lipschitz}.

\begin{lemma}[Differentiability criterion for asymptotes to be timelike]
\label{L:timelike}
Let $(M,g)$ be a strongly causal
spacetime.
Let $b^+=\lim_{r \to \infty} b^+_r$ denote the Busemann function associated by \eqref{b^+_r} to a future-complete $S$-ray $\gamma:[0,\infty)\longrightarrow M$.
Let $\alpha:[0,a_+)\longrightarrow M$ be an asymptote to $\gamma$.
If $b^+$ is differentiable at $\alpha(0)$ and remains Lipschitz nearby,
then $\alpha$ is timelike.
\end{lemma}

\begin{proof}
The definition of asymptote asserts $\alpha$
is a (subsequential) limit curve of a sequence of maximizing geodesics $\sigma_r:[0,s_r] \longrightarrow M$
with $\sigma_r(0)=\alpha(0) \in I^+(S) \cap I^-(\gamma)$ and $\sigma_r(s_r) = \gamma(r)$ as $r \to \infty$. In other words, as $r\to\infty$ along a 
subsequence, $\tilde \sigma_r$ converges uniformly to $\tilde \alpha$ on compact subsets of $[0,\infty)$,
where $\tilde \sigma_r$ and $\tilde \alpha$ denote the $\tilde g$-arclength reparameterizations 
of $\sigma$ and $\alpha$ respectively.
Such an asymptote is a ray by \cite[Lemma 2.4]{GallowayHorta96}, hence $\tilde \alpha$ can be extended to $(-\infty,\infty)$ and affinely reparameterized as a future- and past-inextendible geodesic $\alpha:(a_-,a_+)\longrightarrow M$ for some $a_- \in [-\infty,0)$,  
and $a_+\in (0,\infty]$  
with $\alpha(0)=\tilde \alpha(0)$.

Strong causality of $(M,g)$ and \cite[Thms.\ 2.9, 2.35]{Min:19} imply for $\epsilon>0$ sufficiently small, that $\tilde \alpha(-\epsilon) \in I^-(S)$ and
$\tilde \alpha|_{[-\epsilon,\epsilon]}$ is maximizing.
As in the last paragraph of the proof of 
Proposition~\ref{P:semiconcave} (and of Proposition~\ref{P:asymptotes}), setting 
$x =\alpha(0)$ and $y_r=\tilde \sigma_r(\epsilon)$ yields
$$b^+_r(\tilde \alpha(s)) 
\le b^+_r(x) + \ell(x,y_r) - \ell(\tilde \alpha(s),y_r)
$$
for all $|s|<\epsilon$.  To derive a contradiction, assume $\alpha$ is null.  Letting $r \to \infty$
along the relevant subsequence yields
\begin{align*}
b^+(\tilde \alpha(s)) 
&\le b^+(x) + \ell(x,\tilde \alpha(\epsilon)) - \ell(\tilde \alpha(s),\tilde \alpha(\epsilon))
\\ &= b^+(\tilde \alpha(0))
\end{align*}
by the nullity of $\tilde \alpha$ and 
our choices of $\epsilon$ and $|s|<\epsilon$.
Differentiation at $s=0$ then shows $db^+|_x(\tilde \alpha'(0))=0$, so $db^+|_x$ is vanishing, spacelike or null.  On the other hand, the Lipschitz continuity hypothesized of $b^+$ around $x=\alpha(0)$
yields $db^+|_x$ timelike (and future-directed, with Lorentzian magnitude at least $1$) as in Remark \ref{R:conic intersections}.  This contradiction forces $\alpha$ to be timelike as claimed.
\end{proof}

Our next result is a variant on the semiconcavity Proposition \ref{P:semiconcave}. As before, the equi-semiconcavity of the approximate Busemann functions $(b^+_r)_{r \ge R}$ is essential; semiconcavity of the limiting function $b^+$ alone does not yield the desired corollary.

\begin{proposition}[Equi-semiconcave Busemann limits on asymptotes]\label{P:asymptotes}
Let $(M,g)$ be an (a) globally hyperbolic 
spacetime.
Let $b^+=\lim_{r \to \infty} b^+_r$ denote the Busemann function associated by \eqref{b^+_r} to a future-directed ray $\gamma:[0,\infty)\longrightarrow M$.
Let $\alpha:[0,a_+)\longrightarrow M$ be a 
(future inextendible)
timelike asymptote to $\gamma$. 
Let $b^+ \in C^1(W)$ 
on a neighbourhood of $W$ of 
$x_0=\alpha(a)$ for some $a \in [0, a_+)$. 
Then for large enough 
$C=C(M,g,\tilde g, x_0,db^+(x_0))$ and 
$R$ depending on the same parameters as well as on $\gamma$,
 the functions $(b^+_r)_{r \ge R}$ have semiconcavity constant $C$ on some smaller neighbourhood $X \subset I^-(\gamma(R))$ of $x_0$.
\end{proposition}

\begin{proof}
Let  $\alpha:[0,a_+)\longrightarrow M$ be a (future inextendible, proper-time parameterized) asymptote to $\gamma$ with $b^+$ continuously differentiable in a neighbourhood of $x_0=\alpha(a)$ for some $a\in [0,a_+)$.
Since \cite[Proposition 2.6]{GallowayHorta96} asserts $b^+$ increases at its minimal rate along $\alpha$,  it follows from \eqref{1-steep} that $\alpha'(a)=db^+(x_0)$.
Fixing $0<\epsilon <a_+-a$ yields $y_0:=\alpha(a+\epsilon) = \exp_{x_0} \epsilon db^+(x_0)$.  Since the asymptote is timelike and maximizing,   \cite[Proposition 9.7]{BeemEhrlichEasley96} combines with 
\cite[Theorem 3.6]{McCann20} to
yield a compact neighbourhood $V \subset TM$ of $(x_0,\epsilon db^+(x_0))$ in the timelike future bundle
on which the exponential map is defined and
$\ell$ is smooth on $\exp V$.  
Take $C$ large enough that $C \tilde g$ dominates the $\tilde g$-covariant Hessian (with respect to $x$) of $v(x):=-\ell(x,y)$ for all $(x,y) \in \exp V$.  
We claim $C$ is the desired semiconcavity constant.

Let $X \times Y \subset \exp V$ be a neighbourhood of $(x_0,y_0)$, where $\exp$ here denotes the map taking a vector $v_p$ to $(p,\exp_p(v)) \in M^2$.
Since $b^+(x_0)<\infty$, taking $R$ sufficiently large ensures $x_0 \in I^-(\gamma(R))$ (by the push-up property). Taking $X$ smaller if necessary guarantees $X \subset W \cap I^-(\gamma(R))$ as well. For each $x \in X$ and $r \ge R+1$, global hyperbolicity provides a proper-time parameterized maximizing geodesic segment $\sigma_r:[0,s_r]\longrightarrow M$ from $\sigma_r(0)=x$ to $\sigma_r(s_r)=\gamma(r)$.  The reverse triangle inequality shows $s_r \to \infty$, since $\ell(\gamma(R),\gamma(r)) \to \infty$ as $r \to \infty$.  
Set $y_r=\sigma_r(\epsilon)$.  
By Lemmas 2.1 (the limit curve theorem) and 2.4 of \cite{GallowayHorta96},  one can extract an asymptote $\sigma$ to $\gamma$ as a subsequential limit curve of $\sigma_r$ as $r \to \infty$.  
Since $x \in W$ and $b^+ \in C^1(W)$,  Lemma \ref{L:timelike}
asserts this asymptote is timelike.
Now \cite[Proposition 2.6]{GallowayHorta96} again asserts $b^+$ increases at its minimal rate along $\sigma$,
thus $\sigma(s)=\exp_x sdb^+(x)$ by the same logic as above,  so the full sequence $\sigma_r \to \sigma$ on $[0,a+\epsilon]$ and $y_r \to y$ as $r \to \infty$.  Our hypothesis $b^+ \in C^1$ yields $y \to y_0$ as $x \to x_0$,  so taking $X$ smaller and compact if necessary ensures $y \in Y$ for all $x \in X$ and all $r \ge R+1$.
Since $y_r \to y$ as $r \to \infty$, taking $R$ larger (independently of $x$ within the compact set $X$) then ensures $y_r \in Y$.

For each $x \in X$ and $r \ge R$, Lemma \ref{L:upper support} asserts $u(x') := b^+_r(x) + \ell(x,y_r) - \ell(x',y_r) \ge b^+_r(x')$ for all $x'$ near $x$. Since equality holds at $x'=x$, this means $u(x')$ supports $b^+_r$ from above at $x$.  
Thus $b^+_r$ inherits the asserted semiconcavity constant $C$ 
chosen above from $u$ at $x$.
\end{proof}

\begin{corollary}[$p$-superharmonicity of $b^+$ near asymptotes]
\label{C:asymptotes}
Assume the (a) globally hyperbolic spacetime $(M,g)$ satisfies the strong energy condition \eqref{strong energy}. Then 
inside the neighbourhood $X$ of $\alpha(a)$ identified in Proposition~\ref{P:asymptotes},
$b^+$ is $p$-superharmonic, semiconcave, and $|db^+|=1$.
\end{corollary}

\begin{proof}
Proposition \ref{P:asymptotes} provides the equi-semiconcavity of $\{b^+_r\}_{r \ge R}$ 
necessary for Corollary~\ref{L:harmonicity} to imply $|db^+|=1$ and semiconcavity and $p$-superharmonicity of $b^+$  on the interior of $X$;
global hyperbolicity ensures timelike maximizing geodesics connect each $x \in X \subset I^-(\gamma(R))$ to $\gamma(r)$ for $r \ge R$. 
\end{proof}

We are now ready to conclude the proof of Theorem \ref{T:local splitting} in {(a) the globally hyperbolic} case.

\begin{proof}[Proof of Theorem \ref{T:local splitting} in case (a)]
Recall from the proof of {(b) the timelike geodesically complete} case that in both cases we had identified a neighbourhood $U$ of $\gamma$ on which the Busemann functions $b=b^\pm$ coincide and have vanishing Hessian (hence $db$ is a parallel Killing vector field on $U$),
and a neighbourhood $W$ of $\gamma$
foliated by timelike lines 
$r\in (r^+_x,r^-_x) \mapsto E(r,x)$ which are future- and past-asymptotic to $\gamma$ 
for each $x \in S$,  where $S=U\cap S_0=W \cap S_0$ is spacelike and totally geodesic with timelike unit normal $N=db$ and $S_0 = \{ b=0\}$.  In fact $W$ was the image of an open set $\mathcal{D} \subset \R \times S$ under the smooth diffeomorphism $E(r,s)=\exp_x r db(x)$; moreover,  $b^\pm(E(r,x))=r$ on $\mathcal D$, so $b=b^\pm$ coincide and are smooth
on $W$.  We now argue their Hessians vanish there.

Corollary \ref{C:asymptotes} and its time-reversal
show $\pm b^\pm$ to be $p$-superharmonic on $W$ and $|db^\pm|=1$, thus $b=b^\pm$ is $p$-harmonic and satisfies the conclusions of Proposition \ref{P:maximum principle} on $W$.
Corollary~\ref{C:linear} then asserts that the Hessian of $b$ vanishes on $W$.
Thus $db$ is a parallel Killing vector field throughout $W$ (as was already known on $U$).  This shows
$E$ gives a local isometry between the metrics $dr^2 -h$ on $\mathcal{D}$ and $g$ on $W$,  where $-h$ is the restriction of $g$ to $S$.

Taking 
$U$ (and thus $W$) smaller if necessary ensures $S$ is a geodesic ball centered at $\gamma(0)$.
It remains to deduce that 
$E(\cdot,x)$ is complete for each $x \in S$,
so that $\mathcal{D} = \R \times S$.
This is shown as in Galloway \cite[p.\ 383]{Galloway89}, but we recall the argument for the convenience of the reader; 
we only argue future-completeness. Let $\hat\sigma:[0,R) \to S$ be any radial geodesic starting from $\hat\sigma(0)=\gamma(0)$, where $R$ is the radius of the ball $S$. Let $\alpha_s(\cdot):=E(\cdot,\hat\sigma(s))|_{[0,l_s)}$ be future-inextendible, and $l_s =r^+_{\hat\sigma(s)}\in (0,\infty]$ its Lorentzian arclength.
Fix any $r > R$. We will now show that 
\begin{equation}\label{Galloway's *}
    l_s > r - s
\end{equation}
for all $s \in [0,R)$, thus (by arbitrariness of $r>R$ and $\hat\sigma$) establishing the claimed future completeness of the asymptotic lines to $\gamma$ based in $S$. Denote by $A$ the set of $t \in [0,R)$ such that \eqref{Galloway's *} holds for all $s \in [0,t]$. Clearly, $A$ is an interval containing $0$, since $\alpha_0 = \gamma|_{[0,\infty)}$. Let $a:=\sup A$. We  deduce
$a=R$ by showing the non-empty interval $A$ is both relatively open and closed in $[0,R)$.  
First observe that
since $\mathcal D$ is open (or equivalently, recalling lower semicontinuity of $l_s = r^+_{\hat\sigma(s)}$ from \cite[Proposition 9.7]{BeemEhrlichEasley96}), 
it follows that $A$
is a relatively open subset of $[0,R)$.
On the other hand, we will derive a contradiction by
showing $a< R$ implies $a \in A$. We may also assume $a > 0$ since $0 \in A$ has already been checked. By definition, $l_s > r - s$ for any $s \in [0,a)$. 
The geometry we have established on $W$ shows $\eta(u):=E(r-u,\hat\sigma(u))$ to be a past-directed null geodesic $\eta:[0,a) \to M$ from $\eta(0) = \gamma(r)$. We claim that $l_a > r - a$. If not, then $l_s > r - s > r - a \geq l_a$ for any $s \in [0,a)$. Thus, $\alpha_s(t)$ is well-defined for $s \in [0,a)$ and $t \in [0,l_a)$, and $\alpha_a(t) = \lim_{s \to a} \alpha_s(t)$ by continuity of $E$. It follows that for $s \in [0,a)$,
\begin{equation*}
    \alpha_s(t) \ll \alpha_s(r - s) = E(r-s,\hat\sigma(s)) = \eta(s) \leq \gamma(r).
\end{equation*}
Taking $s \to a$ above and using the closedness of the causal relation guaranteed by global hyperbolicity (a), we conclude that $\alpha_a \subset J^+(\hat\sigma(a)) \cap J^-(\gamma(r))$, a contradiction to non-total imprisonment. Thus $l_a > r -a$ and hence $a \in A$. 
To avoid this contradiction, $a = R$ and \eqref{Galloway's *} holds for all $s \in [0,R)$, to establish the claim.
\end{proof}

\section{Global splitting}

Finally, to globalize the local splitting using connectedness of $M$,  we follow the strategy of 
Eschenburg~\cite{Eschenburg88} (augmented by an observation of Galloway \cite{Galloway89} when {timelike geodesic completeness} (b) fails to hold). We detail the argument for completeness.
Recall a {\em flat strip} refers to a totally geodesic isometric immersion 
$F:(\R \times [0,s_0], dr^2 - ds^2) \longrightarrow (M,g)$ such that $r \in \R \mapsto F(r,s)$ is a (complete) line for any $s \in [0,s_0]$.
Two  such lines $\gamma$ and $\tilde \gamma$ are called {\em strongly parallel} if they bound a flat strip, so that $\gamma(r) =F(r,0)$
and $\tilde \gamma(r) =F(r,s_0)$ for all $r \in \R$ and some $F$ as above. 
To globalize Theorem \ref{T:local splitting},  it is elementary to recall that \cite[Lemma 7.2]{Eschenburg88} 
holds without assuming {global hyperbolicity} (a) or (b):

\begin{lemma}[Strongly parallel lines share their Busemann functions]\label{L:parallel}
Under the hypotheses of Theorem \ref{T:Lipschitz}: 
if $\tilde \gamma$ and $\gamma$ are strongly parallel lines,  then $I(\tilde \gamma)=I(\gamma):=I(\gamma(\R),\gamma(\R))$ 
and their forward Busemann functions $\tilde b^+$ and $b^+$ coincide.
\end{lemma}

\begin{proof}[Proof of Theorem \ref{T:Lorentzian splitting}]
Let $W \subset M$ be the largest connected open subset (ordered by inclusion) on which the conclusion of Theorem \ref{T:local splitting} holds,
meaning $E(r,x) := \exp_x r db^+$ gives a local isometry from $(\R \times S, dr^2 -h)$ onto its image $W$ in $(M,g)$,  
and $r \in \R \mapsto E(r,x)$ is an $S_0$-line for each $x \in S=W \cap S_0$ (and $\gamma(\cdot)=E(\cdot,\gamma(0))$), and $h$ is the restriction of $-g$ to $S$.  
Such a subset exists by Zorn's lemma,  and is non-empty by Theorem \ref{T:local splitting}.
We claim $W$ is a connected component of~$M$.

Whenever $\hat \sigma:[0,s_0] \longrightarrow S$ is an $h$-geodesic in $S$,  then $F(r,s) = E(r,\hat \sigma(s))$ is a flat strip, so its boundaries are strongly parallel.  Lemma~\ref{L:parallel} shows the Busemann functions associated with the lines $E(\cdot,x)$ and $E(\cdot,y)$ through $x=\hat\sigma(0)$
and $y =\hat \sigma(s_0)$ coincide.  
Since $W$ is an open and connected Lorentzian product of $(\R,dr^2)$ with $(S,h)$,
it follows that the hypersurface $S$ is totally geodesic and path-connected.  
Any path in $S$ from $\gamma(0)$ to $x \in S$ can therefore be replaced by a broken geodesic consisting of arbitrarily short geodesic segments, so iterating the argument above yields a finite sequence of lines starting with $\gamma$ and ending with $\beta(\cdot) = E(\cdot, x)$ such that each adjacent pair of lines in the sequence is strongly parallel.  Thus $I(\gamma) = I(\beta)$ and the Busemann functions associated with $\gamma$ and $\beta$
 coincide.

To derive a contradiction, suppose $W$ has a boundary point $y \in \p W$.
Fix a coordinate chart around $y$. Every Euclidean
ball of sufficiently small radius in these coordinates centered near $y$ will be Lorentzian-geodesically convex.  Choose such a ball $B_{2\delta}(y)$ centered at $y$ and then 
$x \in W\cap B_\delta(y)$ sufficiently close to $y$
that the largest Euclidean ball $B_\epsilon(x)$ in $W$ is also Lorentzian-geodesically convex.  This maximality of $\epsilon < 2\delta$ implies there exists $z \in \p B_\epsilon(x) \cap \p W$.  Moreover,  there is a Lorentzian geodesic
$\sigma:[0,s_0] \longrightarrow M$ in $B_\epsilon(x)$
from $x=\sigma(0)$ to $z=\sigma(s_0)$
of the form $\sigma(s) = ((1-s/s_0)r_0 + sr_1/s_0,\hat \sigma(s))$
where $\hat \sigma(s)$ is a (nonconstant) $h$-geodesic in $S$.
For each $s<s_0$,  the preceding paragraph shows
the line $E(\cdot,\hat \sigma(s))$ through $\sigma(s)$ 
shares the same Busemann function
as~$\gamma$.  The product geometry guarantees
that the tangent $X(s)$ to this line at $\sigma(s)$ is the parallel translate along $\sigma$ of the tangent $X(0)$
to the analogous line through $x=\sigma(0)$.
Letting $s\to s_0$,  \cite[Lemma~2.4]{GallowayHorta96} provides
a subsequence of these lines which converge in the limit curve sense to a line through $z=\sigma(s_0)$.
Since the parallel transport $X(s_0)$ of $X(0)$
along $\sigma$ is tangent to this line, and timelike,
we can proper-time reparameterize the limiting line
as $\tilde \gamma:(a_-,a_+)\longrightarrow M$.  By the time-translation symmetry of $W$
we can assume $r_1=0$, and choose $z=\tilde \gamma(0)$ to lie in the closure 
of $S$.
 
In case (b) the line $\tilde \gamma$ is complete: $a_+=\infty=-a_-$.  
Denote its forward Busemann function by 
$\tilde b^+$.  Theorem~\ref{T:local splitting} provides a neighbourhood $\tilde W$ of $\tilde \gamma$
on which $\tilde E(r,x) := \exp_x r d\tilde b$ gives a local isometry $\tilde E:(\R \times \tilde S, dr^2 - \tilde h) \longrightarrow (M,g)$ such that
$r \in \R \mapsto \tilde E(r,x)$ is an $\tilde S_0$-line for each $x \in \tilde S=\tilde W \cap \tilde S_0$ (and $\tilde \gamma(\cdot)=\tilde E(\cdot,\tilde \gamma(0))$),
and $\tilde h$ is the restriction of $-g$ to $\tilde S$. Since the product geometry shows both $S_0$ and $\tilde S_0$ are totally geodesic and orthogonal to $X(s_0)$, they must coincide on the nonempty set $W \cap \tilde W$
--- as must $E$ and $\tilde E$.  Since
$\tilde F(r,s) = \exp_{\sigma(s)} r X(s)$
is a flat strip bounded by $\beta$ and $\tilde \gamma$,  Lemma
\ref{L:parallel} shows $b^+=\tilde b^+$
hence $S_0=\tilde S_0$.
Note for $w \in S$ and $\tilde w \in \tilde S$, the lines $E(\cdot,w)$ and $\tilde E(\cdot,\tilde w)$ cannot cross (unless they coincide), since both are assumed to maximize time to $S_0$. Thus
$W \cup \tilde W$ provides an enlargement of $W$ on which the conclusion of Theorem \ref{T:local splitting} holds, contradicting the assumed maximality of $W$.
This contradiction forces $\p W$ to be empty;
connectedness of $M$ yields $M=W$,
and the splitting becomes global by Corollary \ref{C:local splitting}.
  Ricci nonnegativity of $(S,h)$ follows from the strong energy condition \eqref{strong energy} by the tensorization of the Ricci tensor in 
product geometries, as in \cite[Corollary 7.43]{O'Neill83} with trivial warping factor $f=1$.

The logic and conclusion of the preceding paragraph will apply to {(a) the globally hyperbolic} case  also as soon as
completeness of the timelike line $\tilde \gamma$
is established.  We'll show future-completeness $a_+=\infty$ as in \cite{Galloway89};  past-completeness $a_-=-\infty$ can be shown similarly (or by time-reversal symmetry).
Defining $\tilde F(r,s) = \exp_{\sigma(s)} r X(s)$
as above, it remains true that the restriction of $\tilde F$ to $\R \times [0,s]$ is a flat strip for each $s<s_0$.  Moreover,  for each $t \in (a_-,a_+)$
there is a sequence $(r_i,s_i) \in \R \times [0,s_0)$
such that $\tilde F(r_i,s_i) \to \tilde \gamma(t)$.
Since $\tilde \gamma(t)$ is separated from $\tilde \gamma(0) \in \bar S$ by time $t<a_+$ and $\tilde F(r_i,s_i)$ is separated from $\tilde \gamma(0)$ by time not much less than $r_i + (1-s_i/s_0)r_0$,
continuity of the time-separation function $\ell$
implies $r_i < a_+ + 1$ for $i$ sufficiently large.
Also $R>|r_0| + a_++1 + |\hat \sigma'(0)|_h$
ensures 
$\tilde F(R,0)$ lies in the future of
$\tilde F(r_i,s_i)$ for all $i$ sufficiently large,
by \eqref{local to global}.
Since $t<a_+<\infty$ was arbitrary, the restriction 
$\tilde \gamma|_{[0,a_+)}$ --- being future-inextendible --- is a timelike curve of unbounded $\tilde g$-length in the compact diamond 
$J(\tilde \gamma(0),\tilde F(R,0))$,
contradicting the nontotal imprisonment which global hyperbolicity (a) implies.
We therefore conclude future-
completeness of $\tilde \gamma$:  $a_+=+\infty$ 
(and past-completeness $a_-=-\infty$ similarly).

Apart from (metric) completeness of $(S,h)$,  Theorem \ref{T:Lorentzian splitting} has now been established.
To see completeness of $(S,h)$, let $x_k \in S$ denote a Cauchy sequence.  Then the entire sequence $\{x_k\}_{k \in {\mathbf N}}$ lies in an open $h$-ball $B_r(x_1)$
of radius $r$ sufficiently large. From \eqref{global time separation} it follows that $x_k$ also lies in the diamond $J(E(-r,x_1),E(r,x_1))$, which is compact assuming
(a) global hyperbolicity.  In this case $x_k$ admits a subsequential limit $x_\infty$ in $M$.  We claim $d_h(x_k,x_\infty) \to 0$. This follows from the facts
(i) that $d_h$ metrizes the topology $S$ inherits from $M$,
and (ii) that whenever a Cauchy sequence has a convergent subsequence then the full sequence also converges (to the same limit).

If instead $(M,g)$ is (b) timelike geodesically complete,  we will assume incompleteness of $(S,h)$ to derive a contradiction. In this case the Hopf-Rinow theorem provides an $h$-geodesic $\tau:(s,t) \longrightarrow S$ which is inextendible (say at $t$).
Lifting $\tau$ produces a timelike geodesic $\beta(r):=(r,\tau(r/2))$ in the product metric $dr^2 -h$, which is future-inextendible at $r=2t$:
the desired contradiction to (b).  So $(S,h)$ is complete and Theorem \ref{T:Lorentzian splitting} is established.
\end{proof}

\section*{Acknowledgments}

MB is supported by the EPFL through a Bernoulli Instructorship. A large part of this work was carried out during his Postdoctoral Fellowship at the University of Toronto.  NG is supported in part by the MUR PRIN-2022JJ8KER grant ``Contemporary perspectives on geometry and gravity". RJM's research is supported in part by the Canada Research Chairs program CRC-2020-00289, a grant from the Simons Foundation (923125, McCann), Natural Sciences and Engineering Research Council of Canada Discovery Grant RGPIN- 2020--04162, and Toronto's Fields Institute for the Mathematical Sciences, where this collaboration began. AO is supported in part by the ÖAW scholarship of the Austrian Academy of Sciences. C.S.\ was partly supported by the European Research Council (ERC), under the European’s Union Horizon 2020 research and innovation programme, via the ERC Starting Grant “CURVATURE”, grant agreement No.\ 802689. This research was funded in part by the Austrian Science Fund (FWF) [Grants DOI \href{https://doi.org/10.55776/STA32}{10.55776/STA32}, \href{https://doi.org/10.55776/EFP6}{10.55776/EFP6} and \href{https://doi.org/10.55776/J4913}{10.55776/J4913}]. For open access purposes, the authors have applied a CC BY public copyright license to any author accepted manuscript version arising from this submission. The authors are grateful to Guido De Philippis and Cale Rankin for stimulating exchanges.

\bibliography{newbib} 
\bibliographystyle{abbrv}
\end{document}